\documentclass[10pt,a4paper]{article}

\usepackage{anysize}
\marginsize{3cm}{3cm}{2cm}{2cm}

\usepackage{amsmath,amssymb,amsthm}

\newcommand{\R}{\mathbb{R}}

\newcommand{\N}{\mathbb{N}}
\newcommand{\Z}{\mathbb{Z}}
\renewcommand{\S}{\mathcal{S}}
\newcommand{\F}{\mathcal{F}}

\renewcommand{\i}{\mathbf{i}}

\theoremstyle{plain}
\newtheorem{thm}{Theorem}[section]
\newtheorem{lma}[thm]{Lemma}
\newtheorem{cor}[thm]{Corollary}
\newtheorem{prop}[thm]{Proposition}

\theoremstyle{definition}
\newtheorem{df}[thm]{Definition}

\theoremstyle{remark}
\newtheorem*{rem}{Remark}
\theoremstyle{remark}

\newcommand{\supp}{\mathop{\mathrm{supp}}}
 \newcommand{\sgn}{\mathop{\mathrm{sgn}}}
 
\title{A Non-analytic Superposition Result on Gevrey-modulation Spaces}
\date{\today}

\author{Maximilian Reich \\
\small Fakult\"at f\"ur Mathematik und Informatik \\[-0.8ex]
\small TU Bergakademie Freiberg, Germany\\
\small \texttt{maximilian.reich@math.tu-freiberg.de}\\
}

\begin{document}
	
	\maketitle
	
	\begin{abstract}
		After defining classical weighted modulation spaces we show some basic properties. In this work we additionally choose an approach in terms of the frequency-uniform decomposition and a discussion on the weights of modulation spaces leads to a definition of Gevrey-modulation spaces, where we leave the Sobolev frame and proceed to the Gevrey frame in order to get better results. We prove that Gevrey-modulation spaces are algebras under multiplication. Moreover, we obtain a non-analytic superposition result which gives rise to discuss the possibility to apply Gevrey-modulation spaces to non-linear partial differential equations. 
	\end{abstract}

	\newpage
	
	\begin{flushleft}
		\textbf{\Large{Acknowledgment}} \\[8pt]
	\end{flushleft}
		I gratefully acknowledge the constructive discussions on the field of modulation spaces with Professor Toft from Linnaeus University V\"axj\"o in Sweden. Moreover, I would like to give a special thanks to my supervisor Professor Reissig for his friendly and patient support, his advice and explanations were always very helpful. Additionally I want to express my gratitude to Professor Sickel from Friedrich-Schiller University Jena who gave essential ideas to some results in this work.
	
	\newpage

	\tableofcontents
	\newpage
	
	\section{Introduction}
	
		\subsection{Motivation}
			In signal analysis the goal is to determine the frequency spectrum of a signal at a time $x$. Therefore a signal is recorded over a short period $I=[x-\delta,x]$. This signal is referred to as $f_{inst}$ and is defined by $f_{inst} = f \cdot \chi_I$, where $\chi_I$ is the characteristic function of the interval $I=[x-\delta,x]$. Hence, the frequency spectrum $\hat{f}_{inst}$ is of interest. But it is well-known that an ideal resolution of the frequency at a time $x$ is not possible in general. This obstacle can be reasoned by so-called uncertainty principles. The theory is presented in \cite{groechenig}. Mathematically this means that the support of $\hat{f}_{inst}$ cannot be small. Thus, it is impossible to obtain an instantaneous frequency and it does not make sense to speak of a specific frequency at a time $x$. But of course we want to get a reasonable and fruitful way to determine time-frequency information of a signal which introduces us to the field of time-frequency analysis. At this point a natural link to quantum mechanics can be mentioned. Instead of time and frequency we speak of position and momentum, respectively. Similar arguments yield that we can only find a probability distribution of the position $x$ together with its momentum $\omega$. Hence, it is impossible to assign simultaneously a particle's exact position $x$ to its exact momentum $\omega$ with arbitrary precision. In fact that is precisely what the well-known Heisenberg's uncertainty principle states. \\
			The mathematical model behind the previous considerations is explicitly explained in \cite{groechenig}. We will only give a rough sketch in the following. In \cite{groechenig} it is explained in more detail. Let $f$ be a function. Then its properties are completely obtained by the values $f(x)$ for all $x\in\R^n$. We even get the properties of its Fourier transform $\hat{f}$ since the Fourier transform is one-to-one. But an essential problem arises from the following fact: If $f \in L^p(\R^n)$ is an element of the standard space of Lebesgue integrable functions $L^p$, then we cannot say the same for $\hat{f}$, i.e., we do not know if $\hat{f}\in L^p(\R^n)$ for $p\neq 2$. Thus, the functions $f$ and its Fourier transform $\hat{f}$ are two different representations of the same object but indeed they show different properties of this object. Consequently we search for representations which combine $f$ and $\hat{f}$ and its different features. \\
			The most common joint time-frequency representation is the so-called short-time Fourier transform which already turned out to be rather fruitful, for instance compare with the theory in \cite{groechenig}. It reveals information about local properties of the function $f$. In terms of the discussion above it particularly gives information about the ``local frequency spectrum''. \\
			The next goal consists of finding a family of Banach spaces that are defined by means of the global behavior of certain local properties. Taking the previous considerations into account we want to find a space which controls globally the short-time Fourier transform. This directly leads to modulation spaces which got introduced by Feichtinger in the beginning of 1980s. For more details we refer to \cite{feichtingerHistory}. But his original approach actually based on another idea. Feichtinger realized that modulation spaces basically correspond to so-called Wiener amalgam spaces on the Fourier transform side (see \cite{feichtingerWiener}, \cite{feichtingerBanach}, \cite{feichtingerGroup}). In fact modulation spaces and their Fourier transforms are of the same structure. That points out a big advantage when treating those spaces. Moreover, for spaces of Wiener type there existed already results which have been carried over to modulation spaces in some sense. At this point we only mention properties like duality, multiplier estimates and interpolation methods but there are more statements. \\
			Nowadays modulation spaces are of great interest in time-frequency analysis because of their large number of applications, for instance the modeling of wireless channels, the analysis of linear operators and so on. Also in the theory of pseudo-differential operators modulation spaces are applicable (see \cite{toft}, \cite{toftPseudo}, \cite{toftCubo}). \\
			However our main goal in the future will be to apply modulation spaces to partial differential equations. In this work we will establish a very basic linear result but our main focus is on an introduction to the theory of modulation spaces and eventually the preparation of some tools in order to be able to treat non-linear problems as well. Some results on this field are already existing (see \cite{iwabuchi}, \cite{wang}, \cite{wangNonlinear}, \cite{ruzhansky}).
			
		\subsection{The Short-time Fourier Transform (STFT)}
			\subsubsection{Definition}
			First of all we define an appropriate joint time-frequency representation.
			\begin{df} \label{STFT}
				Let $\phi\neq 0$ be a fixed function, the so-called \emph{window function}. Then the \emph{short-time Fourier transform} (STFT) of a function $f$ with respect to $\phi$ is defined as
				\[ V_\phi f(x,\xi) = (2\pi)^{-\frac{n}{2}} \int_{\R^n} f(s) \overline{\phi(s-x)} e^{-\i s\cdot \xi} ds \qquad (x,\xi \in \R^n). \]
			\end{df}
			It is needed to choose sufficiently smooth window functions to avoid artificial discontinuities of the corresponding STFT $V_\phi f$. What this means for $\phi$ in particular will be shown later on. The window function reveals local properties of the function $f$, that is, we just Fourier transform the function $f$ restricted on an interval, the so-called \emph{window}, determined by the window function $\phi$. Hence we obtain local information about frequency properties of $f$. \\
	Let $\phi$ now be supported on a compact set centered in the origin. Then $V_\phi f(x,\cdot)$ is the Fourier transform of the function $f$ in a neighborhood of $x$. This window can be shifted by choosing different values for $x$. That is why in \cite{groechenig} the STFT is also called the "sliding window Fourier transform". Moreover, the STFT is linear in $f$ and conjugate-linear in $\phi$. In the definition \ref{STFT} we fixed the window function $\phi$, so that the short-time Fourier transform $V_\phi f$ becomes a linear mapping from functions on $\R^n$ to functions on $\R^{2n}$. But obviously $V_\phi f$ also depends essentially on $\phi$. Because of this fact we assumed $\phi$ to be sufficiently smooth. \\
	
			\subsubsection{Function Spaces for STFT}
			We want to investigate more precisely which function spaces for $f$ and $\phi$ are eventually appropriate to define the short-time Fourier transform $V_\phi f$. Therefore we firstly introduce two operators. For $x,\xi\in \R^n$ we define respectively the translation operator $T_x$ and modulation operator $M_\xi$ by
	\[ T_x f(t) = f(t-x) \]
	and 
	\[ M_\xi f(t) = e^{\i \xi\cdot t} f(t). \]
	Naturally we can define the product between both operators $T_x$ and $M_\xi$. The operators we obtain are so-called \emph{time-frequency shifts} $M_\xi T_x$ and $T_x M_\xi$, respectively. \\
	By basic properties of the introduced operators and some straightforward computations we can easily prove the subsequent lemma which gives equal representations of the short-time Fourier transform $V_\phi f$ of a function $f$. Remark that the involution of a function $\phi$ is defined as $\phi^*(x) = \overline{\phi(-x)}$. Furthermore we recall the \emph{Fourier transform} which is defined by 
	\[ \mathcal{F}f(\xi) = \hat{f}(\xi) = (2\pi)^{-\frac{n}{2}}\int_{\R^n} f(x) e^{-\i x\cdot \xi} dx \qquad (x,\xi \in \R^n) \]
	for admissible functions $f$. The \emph{inverse Fourier transform} is defined by
	\[ \mathcal{F}^{-1}\hat{f}(x) = f(x) = (2\pi)^{-\frac{n}{2}}\int_{\R^n} \hat{f}(\xi) e^{\i x\cdot \xi} d\xi. \]
		\begin{lma} \label{STFTrepresent}
			Let $f$ and $\phi$ be admissible functions. Then $V_\phi f$ is uniformly continuous on $\R^{2n}$ and
			\begin{eqnarray}
				\label{eq1}
				V_\phi f(x,\xi) & = & \mathcal{F} (f\cdot T_x \bar{\phi}) (\xi) \\
				\label{eq2}
				& = & (2\pi)^{-\frac{n}{2}} ( f, M_\xi T_x \phi )_{L^2} \\
				\label{eq3}
				& = & (2\pi)^{-\frac{n}{2}} ( \hat{f}, T_\xi M_{-x} \hat{\phi} )_{L^2} \\
				\label{eq4}
				& = & e^{-\i x\cdot \xi} \mathcal{F} (\hat{f}\cdot T_\xi \bar{\hat{\phi}}) (-x) \\
				\label{eq5}
				& = & e^{-\i x\cdot \xi} V_{\hat{\phi}} \hat{f} (\xi, -x) \\
				\label{eq6}
				& = & (2\pi)^{-\frac{n}{2}} e^{-\i x\cdot \xi} (f \ast M_\xi \phi^*) (x) \\
				\label{eq7}
				& = & (2\pi)^{-\frac{n}{2}} (\hat{f}\ast M_{-x}\hat{\phi}^*)(\xi)
			\end{eqnarray}
		\end{lma}
		\begin{proof}
			Cf. Lemma 5.2 in \cite{reich}.
		\end{proof}	
		So far we assumed admissible functions when we considered the short-time Fourier transform. What this in particular means can be basically deduced by Lemma \ref{STFTrepresent}. Due to equality (\ref{eq1}) together with H\"older's inequality the short-time Fourier transform $V_\phi f$ exists pointwise for $f\in L^p(\R^n)$ and $\phi\in L^{p'}(\R^n)$ with $\frac{1}{p}+ \frac{1}{p'} = 1$. A more general existence result can be obtained by equality (\ref{eq2}). If $B$ is a Banach space, then we know by chapter 4 in \cite{rudin} that the dual space $B^*$ exists and is also a Banach space. Additionally we assume that $B$ is invariant under time-frequency shifts. Hence, the expression $( \cdot, \cdot )_{L^2}$ is well defined by duality. Therefore $V_\phi f$ exists for $f\in B, \phi\in B^*$ and for $f\in B^*, \phi \in B$, respectively. Summarizing we can also define the short-time Fourier transform of distributions by taking corresponding test functions as windows. \\
		The following results are proved in \cite{groechenig}. 
		\begin{prop} \label{STFTtemp}
			For a fixed window function $\phi\in\S(\R^n)\setminus \{0\}$ and for $f\in\S'(\R^n)$ the STFT $V_\phi f$ is both defined and continuous on 		$\S'(\R^{2n})$.
		\end{prop}
		\begin{prop} \label{STFTdecay}
			Let $\phi\in \S(\R^n)$ be a fixed window function. If $f\in \S(\R^n)$, then $V_\phi f\in \S(\R^{2n})$. In particular for all $N\geq 0$ there exists a constant $C_N >0$ such that
			\[ |V_\phi f (x,\xi)| \leq C_N (1+|x|+|\xi|)^{-N}. \]
		\end{prop}
		\begin{rem}
			It can be shown that this statement also holds vice versa.
		\end{rem}
		
		\subsubsection{Basic Properties}
		The next proposition gives rise to another interesting question.  
		\begin{prop} \label{orthogonal}
			If $f,\phi \in L^2(\R^n)$, then
			\[ \| V_\phi f \|_{L^2(\R^n)} = \| f \|_{L^2(\R^n)} \| \phi \|_{L^2(\R^n)}. \]
		\end{prop}
		\begin{proof}
			Cf. Corollary 3.2.2 in \cite{groechenig}.
		\end{proof}
		\begin{rem}
			In particular, if $\| \phi \|_{L^2(\R^n)} = 1$, then the STFT is an isometry from $L^2(\R^n)$ into $L^2(\R^{2n})$ and it holds
			\[ \| V_\phi f \|_{L^2(\R^n)} = \| f \|_{L^2(\R^n)}. \] 
		\end{rem}
		Hence, the short-time Fourier transform $V_\phi f$ determines the function $f$ completely. This means if $V_\phi f(x,\xi) = (2\pi)^{-\frac{n}{2}} ( f, M_\xi T_x \phi )_{L^2} = 0$ for all $(x,\xi) \in \R^{2n}$, then $f=0$ on $\R^n$. Thus, it is natural to ask for an inversion formula as it is known for the usual Fourier transform. The fundamental theory justifying these considerations is presented in chapter 3 in \cite{rudin}. 
		\begin{thm} \label{inversion}
			Let $\phi,\gamma \in L^2(\R^n)$ such that $\langle \phi, \gamma \rangle \neq 0$. Then 
			\[ f = (2\pi)^{-\frac{n}{2}} \frac{1}{( \phi, \gamma )_{L^2}} \iint_{\R^{2n}} V_\phi f(x,\xi) M_\xi T_x \gamma \, d\xi \, dx \]
			for every $f\in L^2(\R^n)$.
		\end{thm}
		\begin{proof}
			Cf. Corollary 3.2.3 in \cite{groechenig}.
		\end{proof}
		\begin{rem}
			In Corollary 11.2.7 in \cite{groechenig} it is shown that the inversion formula also holds in $\S'(\R^n)$. 
		\end{rem}
		After we established the so-called inversion formula of time-frequency analysis the existence of an adjoint $V_\phi^*$ of $V_\phi$ is of great interest. Therefore we define a linear operator $A_\phi$ by 
		\[ A_\phi F = (2\pi)^{-\frac{n}{2}} \iint_{\R^{2n}} F(x,\xi) M_\xi T_x \phi \, dx\, d\xi, \qquad F\in L^2(\R^{2n}) \]
		with an admissible window function $\phi\in L^2(\R^n)\setminus \{0\}$. In fact, the operator $A_\phi$ is the adjoint of $V_\phi$ if we consider the short-time Fourier transform as a map from $L^2(\R^n)$ to $L^2(\R^n)$. Thus, $A_\phi = V_\phi^*$. From Theorem \ref{inversion} we can deduce 
		\[ ( f,h )_{L^2} = \frac{1}{( \phi, \gamma )_{L^2}} ( V_\gamma^* V_\phi f,h )_{L^2} \]
		and thus 
		\begin{equation} \label{adjointId}
			\frac{1}{( \phi, \gamma )_{L^2}} V_\gamma^* V_\phi = I, 
		\end{equation}
		where $I$ is the identity operator on $L^2(\R^n)$.\\
		By now we worked with admissible functions to show the concepts of how to get the basic properties of the short-time Fourier transform. These results can be naturally extended to distributions. \\
		Another helpful result is the following lemma.
	\begin{lma} \label{twisted}
		If $\phi_0,\phi,\gamma \in \S(\R^n)\setminus \{0\}$ such that $\langle \gamma,\phi \rangle \neq 0$ and $f\in\S'(\R^n)$, then it holds
		\[ |V_{\phi_0} f(x,\xi)| \leq (2\pi)^{-\frac{n}{2}} \frac{1}{|( \gamma,\phi )_{L^2}|} ( |V_\phi f| \ast |V_{\phi_0} \gamma|) (x,\xi) \]
		for all $(x,\xi)\in\R^{2n}$. 
	\end{lma}
	\begin{proof}
		Cf. Lemma 16 in \cite{reich}.
	\end{proof}
		
		\subsubsection{Alternative Approach} \label{altApproach}
		Recalling that the idea of the short-time Fourier transform in fact was to obtain local frequency properties of a function $f$ by taking the Fourier transform in a so-called window. As a window we chose a sufficiently smooth function $\phi$ to avoid discontinuities. By roughly adopting this idea we can establish the following approach to the STFT. \\
		The so-called \textit{frequency-uniform decomposition} gives rise to find another definition of modulation spaces. For that let $\rho: \R^n \mapsto [0,1]$ be a Schwartz function which is compactly supported in the cube $Q_0 := \{ \xi \in \R^n: -1\leq \xi_i \leq 1, i=1,\ldots,n \}$. Moreover, $\rho(\xi)=1$ if $|\xi|\leq \frac{1}{2}$. Naturally one obtains the shifted functions $\rho_k(\xi) = \rho(\xi-k)$ for $k\in \Z^n$. Finally we define
		\[ \sigma_k(\xi) = \rho_k(\xi) \left(\sum_{k\in\Z^n} \rho_k(\xi)\right)^{-1}, \quad k\in\Z^n \]
		with the following obvious properties:
		\begin{itemize}
			\item $|\sigma_k(\xi)| \geq C$ for all $\xi\in B(k, \frac{1}{2}):= \{ \xi\in\R^n : |\xi-k|\leq \frac{1}{2} \}$;
			\item $\supp \sigma_k \subset Q_k := \{ \xi \in \R^n: -1\leq \xi_i -k_i \leq 1, i=1,\ldots,n \} \subset B(k,\sqrt{n})$;
			\item $\displaystyle \sum_{k\in\Z^n} \sigma_k(\xi) \equiv 1$ for all $\xi\in\R^n$;
			\item $|D^\alpha \sigma_k(\xi)|\leq C_m$ for all $\xi\in\R^n$ and $|\alpha|\leq m$.
		\end{itemize}
		The operator
		\[ \Box_k := \F^{-1} \left( \sigma_k \F (\cdot) \right), \quad k\in\Z^n \]
		is called \textit{uniform decomposition operator}. Now the similarity to Definition \ref{STFT} of the STFT is obvious. Taking Lemma \ref{STFTrepresent} into account we have
		\begin{eqnarray*}
			\Box_k f(x) & = & \F^{-1} \left( \sigma_k \hat{f} \right)(x) \\
				& = & \int_{\R^n} \hat{f}(\eta) \sigma(\eta-k) e^{\i\eta\cdot x} \, d\eta \\
				& = & (2\pi)^{\frac{n}{2}} (V_\sigma \hat{f}) (k, -x) \\
				& = & (2\pi)^{\frac{n}{2}} e^{\i k\cdot x} (V_{\hat{\sigma}} f)(-x, -k).
		\end{eqnarray*}
		Note that $k\in\Z^n$, i.e., the frequency-uniform decomposition handles discrete frequencies. 
		
		\subsection{Modulation Spaces}
			So far we found a joint time-frequency representation of a function $f$ namely its short-time Fourier transform $V_\phi f$. The goal was to get information about the behavior of a function and its Fourier transform at the same time. After we obtained those information we naturally want to control them in some sense. Therefore we introduce weighted modulation spaces. A detailed concept of \emph{weights} can be found in chapter 11 in \cite{groechenig}. Subsequently however we will only use particular weights. 
		\begin{df} \label{modCont}
		The so-called \emph{integrability parameters} are given by $1\leq p,q\leq\infty$. Let $\phi\in \S(\R^n)\setminus \{0\}$ be a fixed window and assume $s,\sigma \in\R$ to be the weight parameters. Then the \emph{weighted modulation space} $\mathring{M}^{p,q}_{s,\sigma}(\R^n)$ is the set
		\[ \mathring{M}^{p,q}_{s,\sigma} (\R^n) := \{ f\in \S'(\R^n): \|f\|_{\mathring{M}^{p,q}_{s,\sigma}(\R^n)} < \infty \}, \]
		where the norm is defined as
		\[ \|f\|_{\mathring{M}^{p,q}_{s,\sigma}(\R^n)} = \left( \int_{\R^n} \left( \int_{\R^n} |V_\phi f(x,\xi) \langle x\rangle^\sigma \langle \xi\rangle^s|^p dx \right)^{\frac{q}{p}} d\xi \right)^{\frac{1}{q}}. \]
		Furthermore, the weighted modulation space $W^{p,q}_{s,\sigma}(\R^n)$ consists of all tempered distributions $f\in \S'(\R^n)$ such that their norm
		\[ \|f\|_{W^{p,q}_{s,\sigma}(\R^n)} = \left( \int_{\R^n} \left( \int_{\R^n} |V_\phi f(x,\xi) \langle x\rangle^\sigma \langle \xi\rangle^s|^q d\xi \right)^{\frac{p}{q}} dx \right)^{\frac{1}{p}} \]
		is finite. \\
		For $p=\infty$ and/or $q=\infty$ the definition can be obviously modified by taking $L^\infty$ norms. 
	\end{df}
	\begin{rem}
		Note that 
		\[ \langle x\rangle^\sigma = (1+|x|^2)^{\frac{\sigma}{2}} \qquad \mbox{and} \qquad  \langle \xi\rangle^s = (1+|\xi|^2)^{\frac{s}{2}}. \]
		If $s=\sigma=0$ then we obtain the so-called \emph{standard modulation space} $\mathring{M}^{p,q}(\R^n)$, that is the modulation space without any weights. If we only have $\sigma=0$, i.e., no weight with respect to $x$-variable, then the weighted modulation space is denoted by $\mathring{M}^{p,q}_s(\R^n)$. Subsequently the space $\mathring{M}^{p,q}_{s,\sigma}(\R^n)$ is just referred to as modulation space. Furthermore if $p=q$ we write $\mathring{M}^p_{s,\sigma}(\R^n)$ instead of $\mathring{M}^{p,p}_{s,\sigma}(\R^n)$. \\
		The same notations apply to the modulation space $W^{p,q}_{s,\sigma}(\R^n)$. Additionally all following results hold analogously for $W^{p,q}_{s,\sigma}(\R^n)$. 
	\end{rem}
	Note that the weight expression with respect to $x$ in the preceding definition corresponds to some growth or decay properties of a function $f$ in the modulation space $\mathring{M}^{p,q}_{s,\sigma}$. On the other hand the weight expression with respect to $\xi$ corresponds to regularity properties of $f$ in $\mathring{M}^{p,q}_{s,\sigma}$. The following proposition shows these facts in a mathematically more precise way. Here we recall that $D_j=\frac{1}{\i} \frac{\partial}{\partial x_j}$ for $1\leq j \leq n$.
	\begin{prop} \label{ModHomeo}
		Let $s,s_0,\sigma,\sigma_0 \in \R$ and $1\leq p,q\leq \infty$ be the integrability parameters. Then it holds:
		\begin{itemize}
			\item the map $f \mapsto \langle \cdot \rangle^{\sigma_0} f$ is a homeomorphism from $\mathring{M}^{p,q}_{s,\sigma+\sigma_0}(\R^n)$ to $\mathring{M}^{p,q}_{s,\sigma}(\R^n)$ and
			\item the map $f \mapsto \langle D \rangle^{s_0} f$ is a homeomorphism from $\mathring{M}^{p,q}_{s+s_0,\sigma}(\R^n)$ to $\mathring{M}^{p,q}_{s,\sigma}(\R^n)$.
		\end{itemize}		
	\end{prop}
	\begin{proof}
		Cf. Corollary 3.3 in \cite{toftWeight}.
	\end{proof}
	Summarizing Definition \ref{modCont} we imposed on the short-time Fourier transform of a function $f$ some $L^p$ and $L^q$ behavior, respectively. However in Section \ref{altApproach} we found an alternative approach to the STFT. We will prove that defining modulation spaces with the help of the uniform decomposition operator is also reasonable. In particular it is equivalent.
	\begin{df} \label{defdecomp}
		Let $1\leq p,q \leq \infty$ and assume $s \in\R$ to be the weight parameter. Then the \emph{weighted modulation space} $M^{p,q}_{s}(\R^n)$ consists of all tempered distributions $f\in \S'(\R^n)$ such that their norm
		\[ \|f\|_{M^{p,q}_{s}(\R^n)} = \left( \sum_{k\in\Z^n} \langle k \rangle^{sq} \|\Box_k f\|_{L^p}^q \right)^{\frac{1}{q}} \]
		is finite with obvious modifications when $p=\infty$ and/or $q=\infty$.
	\end{df}
	
	In order to prove Proposition \ref{normequivalence} we need the so-called \textit{Bernstein's multiplier estimate} which is stated in \cite{wang}.
	\begin{lma} \label{bernstein}
			Assume that $s>\frac{n}{2}$. Then there exists a constant $C>0$ such that
			\[ \|\F^{-1} \left( \phi\F f \right)\|_{L^r} \leq C \|\phi\|_{H^s} \|f\|_{L^r} \]
			for all $f\in L^r(\R^n)$ and $\phi \in H^s(\R^n)$.
	\end{lma}
	\begin{rem}
			As mentioned in \cite{wang} this lemma also holds for $s>n \left(\frac{1}{\min(r,1)} - \frac{1}{2} \right)$ with $0<r<1$ assuming that $f\in L^r_\Omega := \{f\in L^r: \supp \hat{f} \subset \Omega\}$, where $\Omega\subset \R^n$ is a compact set.
	\end{rem}
	
	\begin{prop} \label{normequivalence}
		The norms of the Definitions \ref{modCont} and \ref{defdecomp} are equivalent. For admissible functions $f$ it holds
		\[ C_1 \|f \|_{\mathring{M}^{p,q}_{s}(\R^n)} \leq \|f\|_{M^{p,q}_{s}(\R^n)} \leq C_2 \|f\|_{\mathring{M}^{p,q}_{s}(\R^n)}, \] 
		where the positive constants $C_1$ and $C_2$ are depending on the dimension $n$. Furthermore they are depending on the window function and on the frequency-uniform decomposition functions, respectively. 
	\end{prop}
	\begin{proof}
		The idea of the proof for finite integrability parameters $p$ and $q$ is given in \cite{wang}. \\		
		By \cite{groechenig} we get 
		\begin{eqnarray*}
			V_\phi f(x,\xi) & = & e^{-\i x\cdot\xi} V_{\hat{\phi}} \hat{f} (\xi, -x) \\
				& = & e^{-\i x\cdot\xi} \int_{\R^n} e^{\i x\cdot\omega} \overline{\hat{\phi}(\omega-\xi)} \hat{f}(\omega) \, d\omega \\
				& = & e^{-\i x\cdot\xi} \F^{-1} \left( \overline{\hat{\phi}(\cdot - \xi)} \hat{f}(\cdot) \right) (x)
		\end{eqnarray*}
		with an admissible window function $\phi\in \S(\R^n)$. \\
		First suppose that $1\leq p,q<\infty$ and $f\in \S(\R^n)$. Due to the mean value theorem there exists a $\xi_k \in Q_k$ for each $k\in \Z^n$ such that
		\begin{eqnarray*}
			\|f\|_{\mathring{M}^{p,q}_s} & = & \left( \int_{\R^n} \langle \xi \rangle^{sq} \|\F^{-1} \left( \overline{\hat{\phi}(\cdot - \xi)} \hat{f}(\cdot) \right) \|^q_{L^p} \, d\xi \right)^{\frac{1}{q}} \\
				& \sim & \left( \sum_{k\in\Z^n} \langle k \rangle^{sq} \|\F^{-1} \left( \overline{\hat{\phi}(\cdot - \xi_k)} \hat{f}(\cdot) \right) \|^q_{L^p} \right)^{\frac{1}{q}}.
		\end{eqnarray*}
		Assuming $\supp \hat{\phi}(\cdot - \xi_k) \subset B(\xi_k,100\sqrt{n})$ and $\hat{\phi}(\cdot-\xi_k)(\xi) = 1$ on $B(\xi_k,3\sqrt{n})$ we obtain 
		\begin{eqnarray*}
			\|\Box_k f\|_{L^p} & = & \| \F^{-1} ( \sigma_k \hat{f} ) \|_{L^p} \\
				& = & \|\F^{-1} \left( \sigma_k \hat{\phi}(\cdot - \xi_k) \hat{f} \right) \|_{L^p} \\
				& = & \| \F^{-1} \Big( \sigma_k \F(\phi(\cdot - x_k) \ast f) \Big) \|_{L^p} \\
				& \lesssim & \|\phi(\cdot - x_k) \ast f\|_{L^p} \\
				& = & \|\F^{-1} \left( \hat{\phi}(\cdot - \xi_k) \hat{f} \right) \|_{L^p}.
		\end{eqnarray*}
		Here we used Lemma \ref{bernstein} and the density of the Schwartz space $\S(\R^n)$ in the modulation space $M^{p,q}_s(\R^n)$. Hence,
		\[ \|f\|_{M^{p,q}_s} \lesssim \|f\|_{\mathring{M}^{p,q}_s}. \]
		For the second part we take into consideration that $\supp \hat{\phi}(\cdot - \xi_k)$ overlaps at most $\mathcal{O}(\sqrt{n})$ many supports of $\sigma_k$. Let $\Lambda$ be a set which contains at most $\mathcal{O}(\sqrt{n})$ many elements. Then we get
		\begin{eqnarray*}
			\| \F^{-1} \left( \hat{\phi}(\cdot -\xi_k) \hat{f}(\cdot) \right) \|_{L^p} & = & \left\| \F^{-1} \Big( \sum_{l\in \Lambda} \sigma_{k+l} \hat{\phi}(\cdot - \xi_k) \hat{f} \Big) \right\|_{L^p} \\
				& = & \left\| \F^{-1} \Bigg( \hat{\phi}(\cdot - \xi_k) \F \Big( \F^{-1} (\sum_{l\in\Lambda} \sigma_{k+l} ) \ast f \Big) \Bigg) \right\|_{L^p} \\
				& \lesssim & \left\| \F^{-1} \Big(\sum_{l\in \Lambda} \sigma_{k+l} \Big) \ast f \right\|_{L^p} \\
				& = & \left\| \F^{-1} \Big( \sum_{l\in\Lambda} \sigma_{k+l} \hat{f} \Big) \right\|_{L^p} \\
				& \leq & \sum_{l\in\Lambda} \left\| \F^{-1} \left( \sigma_{k+l} \hat{f} \right) \right\|_{L^p} \\
				& = & \sum_{l\in\Lambda} \|\Box_{k+l} f\|_{L^p}
		\end{eqnarray*}
		by using again Lemma \ref{bernstein} and density arguments. It follows
		\[ \|f\|_{\mathring{M}^{p,q}_s} \lesssim \|f\|_{M^{p,q}_s} \]
		for $1\leq p,q <\infty$. \\
		In the next step assume $p=\infty$ and $q>1$. Moreover, let $g \in \mathring{M}^{1,q'}_{-s}(\R^n)$, where $\frac{1}{q}+\frac{1}{q'}=1$. By duality which is shown in Theorem \ref{Mduality}, H\"older's inequality and support properties we obtain
		\begin{eqnarray*}
			|(f,g)_{L^2}| & \leq & \sum_{j,k\in \Z^n} | (\Box_j f, \Box_k g)_{L^2} | \\
				& \leq & \sum_{j,k \in \Z^n} \int_{\R^n} |\Box_j f(x)| |\Box_k g(x)| \, dx \\
				& = & \sum_{\substack{ j,k\in \Z^n, \\ -2\leq j_i -k_i \leq 2 }} \int_{\R^n} |\Box_j f(x)| |\Box_k g(x)| \, dx \\
				& \leq & C_1 \sum_{\substack{ j,k\in \Z^n, \\ -2\leq j_i -k_i \leq 2 }} \langle k \rangle^{-s} \langle j \rangle^s \int_{\R^n} |\Box_j f(x)| |\Box_k g(x)| \, dx \\
				& \leq & C_1 \sum_{\substack{ j,k\in \Z^n, \\ -2\leq j_i -k_i \leq 2 }} \left( \langle k \rangle^{-s} \| \Box_k g(x) \|_{L^1} \right) \left( \langle j \rangle^s \| \Box_j f(x) \|_{L^\infty } \right) \\
				& \leq & C_1 \Bigg( \sum_{\substack{ j,k\in \Z^n, \\ -2\leq j_i -k_i \leq 2 }} \langle k \rangle^{-sq'} \| \Box_k g(x) \|^{q'}_{L^1} \Bigg)^{\frac{1}{q'}} \Bigg( \sum_{\substack{ j,k\in \Z^n, \\ -2\leq j_i -k_i \leq 2 }} \langle j \rangle^{sq} \| \Box_j f(x) \|^q_{L^\infty } \Bigg)^{\frac{1}{q}} \\
				& \leq & C_2 \Bigg( \sum_{k\in \Z^n} \langle k \rangle^{-sq'} \| \Box_k g(x) \|^{q'}_{L^1} \Bigg)^{\frac{1}{q'}} \Bigg( \sum_{j\in \Z^n} \langle j \rangle^{sq} \| \Box_j f(x) \|^q_{L^\infty} \Bigg)^{\frac{1}{q}} \\
				& = & C_2 \| g\|_{M^{1,q'}_{-s} } \| f \|_{M^{\infty, q}_{s}} \\
				& \leq & C_3 \|g\|_{\mathring{M}^{1,q'}_{-s}(\R^n)} \| f \|_{M^{\infty, q}_{s}} \\
				& \leq & C_4 \| f \|_{M^{\infty, q}_{s} }.
		\end{eqnarray*}
		This computation yields $M^{\infty, q}_{s} \subset \mathring{M}^{\infty,q}_{s}$ if $q>1$. \\
		Let $\phi \in \S(\R^n)\setminus \{0 \}$ be an admissible window function. By using Lemma \ref{STFTrepresent}, Lemma \ref{twisted} and Lemma 4.2 in \cite{brs} the opposite inclusion is obtained as follows
		\begin{eqnarray*}
			\| f\|^q_{M^{\infty,q}_{s}} & = & \sum_{k\in \Z^n} \langle k \rangle^{sq} \|\Box_k f(\cdot) \|^q_{L^\infty} \\
				& = & \sum_{k\in \Z^n} \langle k \rangle^{sq} \left\| \int_{\R^n} \sigma (\eta -k) \hat{f}(\eta) e^{\i \eta\cdot x} \, d\eta \right\|^q_{L_x^\infty} \\
				& = & (2\pi)^{q\frac{n}{2}} \sum_{k\in \Z^n} \langle k \rangle^{sq} \| (V_\sigma \hat{f}) (k,-\cdot) \|^q_{L^\infty} \\
				& \leq & C_1 \sum_{k\in \Z^n} \langle k \rangle^{sq} \| (V_{\hat{\sigma}} f) (\cdot ,k) \|^q_{L^\infty} \\
				& \leq & C_2 \sum_{k\in \Z^n} \langle k \rangle^{sq} \sup_{x\in \R^n} [ (|V_\phi f| \ast |V_{\hat{\sigma}} \phi|) (x ,k) ]^q \\
				& \leq & C_2 \sum_{k\in \Z^n} \sup_{x\in \R^n} \left[ \int_{\R^n} \int_{\R^n} \langle k-\eta \rangle^{s} |(V_{\hat{\sigma}} \phi)(x-y, k-\eta)|  |(V_\phi f)(y,\eta)| \langle \eta \rangle^{s} \, dy \, d\eta \right]^q \\
				& \leq & C_2 \sum_{k\in \Z^n} \sup_{x\in \R^n} \left[ \int_{\R^n} \langle k-\eta \rangle^{s} \| (V_{\hat{\sigma}} \phi)(x-\cdot, k-\eta) \|_{L^1} \| \langle \eta \rangle^s (V_\phi f)(\cdot,\eta) \|_{L^\infty} \, d\eta \right]^q \\
				& = & C_2 \sum_{k\in \Z^n} \sup_{x\in \R^n} \Bigg[ \int_{\R^n} \left( \langle k-\eta \rangle^{s} \| (V_{\hat{\sigma}} \phi)(x-\cdot, k-\eta) \|_{L^1} \right)^{\frac{1}{q'}} \\
				& & \qquad \left( \langle k-\eta \rangle^{\frac{s}{q}} \| (V_{\hat{\sigma}} \phi)(x-\cdot, k-\eta) \|^{\frac{1}{q}}_{L^1} \| \langle \eta \rangle^s (V_\phi f)(\cdot,\eta) \|_{L^\infty} \right) \, d\eta \Bigg]^q \\
				& \leq & C_2 \sum_{k\in \Z^n} \sup_{x\in \R^n} \Bigg[  \left( \int_{\R^n} \langle k-\eta \rangle^{s} \| (V_{\hat{\sigma}} \phi)(x-\cdot, k-\eta) \|_{L^1} \, d\eta \right)^{\frac{1}{q'}} \\
				& & \qquad \left( \int_{\R^n} \langle k-\eta \rangle^{s} \| (V_{\hat{\sigma}} \phi)(x-\cdot, k-\eta) \|_{L^1} \| \langle \eta \rangle^{s} (V_\phi f)(\cdot,\eta) \|^q_{L^\infty} \, d\eta \right)^{\frac{1}{q}} \Bigg]^q \\
				& \leq & C_3 \sum_{k\in \Z^n} \sup_{x\in \R^n} \left( \int_{\R^n} \langle k-\eta \rangle^{s} \| (V_{\hat{\sigma}} \phi)(x-\cdot, k-\eta) \|_{L^1} \| \langle \eta \rangle^{s} (V_\phi f)(\cdot,\eta) \|^q_{L^\infty} \, d\eta \right) \\
				& \leq & C_3 \sum_{k\in \Z^n} \sup_{x, \eta \in \R^n} \big[ \langle k-\eta \rangle^{s} \| (V_{\hat{\sigma}} \phi)(x-\cdot, k-\eta) \|_{L^1} \big] \int_{\R^n} \langle \eta \rangle^{sq} \|(V_\phi f)(\cdot,\eta) \|^q_{L^\infty} \, d\eta \\
				& \leq & C_4 \|f\|^q_{\mathring{M}^{\infty, q}_{s}},
		\end{eqnarray*}
		where $f\in \mathring{M}^{\infty,q}_{s}(\R^n)$. Thus, we also have $\mathring{M}^{\infty,q}_{s} \subset M^{\infty, q}_{s}$ if $q>1$. \\
		Now it is left to show that $\mathring{M}^{p,\infty}_{s} = M^{p, \infty}_{s}$ if $p>1$. Note that $\frac{1}{p}+\frac{1}{p'} =1$. Let $g \in \mathring{M}^{p',1}_{-s}(\R^n)$. Again due to duality which is stated in Theorem \ref{Mduality}, H\"older's inequality and support properties we deduce 
		\begin{eqnarray*}
			|(f,g)_{L^2}| & \leq & \sum_{\substack{ j,k\in \Z^n, \\ -2\leq j_i -k_i \leq 2 }} \int_{\R^n} |\Box_j f(x)| |\Box_k g(x)| \, dx \\
				& \leq & C_1 \sum_{\substack{ j,k\in \Z^n, \\ -2\leq j_i -k_i \leq 2 }} \left( \langle k \rangle^{-s} \| \Box_k g(x) \|_{L^{p'}} \right) \left( \langle j \rangle^s \|\Box_j f(x) \|_{L^p} \right) \\
				& \leq &  C_1 \sup_{j\in \Z^n} \left( \langle j \rangle^s \|\Box_j f(x) \|_{L^p} \right) \sum_{\substack{ j,k\in \Z^n, \\ -2\leq j_i -k_i \leq 2 }} \left( \langle k \rangle^{-s} \| \Box_k g(x) \|_{L^{p'}} \right) 
		\end{eqnarray*}
		\begin{eqnarray*}
				& \leq & C_2 \sup_{j\in \Z^n} \left( \langle j \rangle^s \|\Box_j f(x) \|_{L^p} \right) \sum_{k\in \Z^n} \left( \langle k \rangle^{-s} \| \Box_k g(x) \|_{L^{p'}} \right) \\
				& = & C_2 \| g\|_{M^{p',1}_{-s}} \| f \|_{M^{p, \infty}_{s}} \\
				& \leq & C_3 \| g\|_{\mathring{M}^{p',1}_{-s}(\R^n)} \| f \|_{M^{p, \infty}_{s}} \\
				& \leq & C_4 \| f \|_{M^{p, \infty}_{s} }.
		\end{eqnarray*}
		Therefore it holds $M^{p, \infty}_{s} \subset \mathring{M}^{p, \infty}_{s}$ if $p>1$. \\
		Last step of the proof consists of showing the opposite inclusion. Let $\phi \in \S(\R^n)\setminus \{0 \}$ be an admissible window function. Recall that the short-time Fourier transform is shift-invariant. Using again Lemma \ref{STFTrepresent}, Lemma \ref{twisted}, Lemma 4.2 in \cite{brs} and Minkowski's inequality we obtain
		\begin{eqnarray*}
			\| f\|_{M^{p, \infty}_{s}} & = & \sup_{k\in \Z^n} \langle k \rangle^{s} \|\Box_k f(\cdot) \|_{L^p} \\
				& \leq & C_1 \sup_{k\in \Z^n} \langle k \rangle^{s} \| (V_{\hat{\sigma}} f) (x,k) \|_{L^p} \\
				& \leq & C_2 \sup_{k\in \Z^n} \left( \int_{\R^n_x} \left| \int_{\R^n_\eta} \int_{\R^n_y} \langle \eta \rangle^s |(V_{\hat{\sigma}} \phi)(y,\eta)| |(V_\phi f)(x-y,k-\eta)| \langle k-\eta \rangle^s \, dy \, d\eta \right|^p \, dx \right)^{\frac{1}{p}} \\
				& \leq & C_2 \sup_{k\in \Z^n} \int_{\R^n_\eta} \int_{\R^n_y} \left( \int_{\R^n_x} \left| \langle \eta \rangle^s |(V_{\hat{\sigma}} \phi)(y,\eta)| |(V_\phi f)(x-y,k-\eta)| \langle k-\eta \rangle^s \right|^p \, dx \right)^{\frac{1}{p}} \, dy \, d\eta \\
				& = & C_2 \sup_{k\in \Z^n} \int_{\R^n_\eta} \int_{\R^n_y}  \langle \eta \rangle^s |(V_{\hat{\sigma}} \phi)(y,\eta)| \left( \int_{\R^n_x} \left( |(V_\phi f)(x-y,k-\eta)| \langle k-\eta \rangle^s \right)^p \, dx \right)^{\frac{1}{p}} \, dy \, d\eta \\
				& = & C_2 \sup_{k\in \Z^n} \int_{\R^n_\eta} \int_{\R^n_y}  \langle \eta \rangle^s |(V_{\hat{\sigma}} \phi)(y,\eta)| \langle k-\eta \rangle^s \| (V_\phi f)(\cdot-y,k-\eta) \|_{L^p} \, dy \, d\eta \\
				& \leq & C_2 \sup_{k\in \Z^n} \int_{\R^n_\eta} \int_{\R^n_y} |(V_{\hat{\sigma}} \phi)(y,\eta)| \langle k-\eta \rangle^s \, dy \, \langle \eta \rangle^s \| (V_\phi f)(\cdot,k-\eta) \|_{L^p} \, d\eta \\
				& \leq & C_2 \sup_{k\in \Z^n} \left( \int_{\R^n_\eta} \int_{\R^n_y}  \langle k-\eta \rangle^s |(V_{\hat{\sigma}} \phi)(y,\eta)| \, dy \, d\eta \right) \Big\| \langle \eta \rangle^s \| (V_\phi f)(x,k-\eta) \|_{L_x^p} \Big\|_{L^\infty_\eta} \\
				& \leq & C_3 \sup_{\eta\in\R^n} \langle \eta \rangle^s \|(V_\phi f)(\cdot,\eta) \|_{L^p} \\
				& = & C_3 \|f \|_{\mathring{M}^{p, \infty}_{s}},
		\end{eqnarray*}
		where $f\in \mathring{M}^{p, \infty}_{s}(\R^n)$. Thus, we also have $\mathring{M}^{p, \infty}_{s} \subset M^{p,\infty}_{s}$ if $p>1$. \\
		At this point only the limit cases $p=1$, $q=\infty$ and $p=\infty$, $q=1$ are left, respectively. The main tool in order to prove it is Proposition 1.4 (3) in \cite{toftConv}. It holds
		\begin{equation} \label{dual1}
			\|f\|_{M^{1,\infty}_s} \lesssim \sup_{\substack{g\in M^{\infty,1}_{-s}, \\ \|g\|_{M^{\infty,1}_{-s}} \leq 1}} |(f,g)_{L^2}|
		\end{equation}
		and
		\begin{equation} \label{dual2}
			\|f\|_{M^{\infty,1}_s} \lesssim \sup_{\substack{g\in M^{1,\infty}_{-s}, \\ \|g\|_{M^{1,\infty}_{-s}} \leq 1}} |(f,g)_{L^2}|
		\end{equation}
		which can be justified by a construction of sequences of admissible functions. In fact this yields that the proposition holds for all $p,q\in [1,\infty]$ since the inclusions $\mathring{M}^{p, \infty}_{s} \subset M^{p, \infty}_{s}$ and $\mathring{M}^{\infty,q}_{s} \subset M^{\infty,q}_{s}$ can be extended to $p=1$ and $q=1$ without any problems, respectively. The opposite inclusions were shown by duality. By (\ref{dual1}) and (\ref{dual2}) we can apply the same arguments as above for the limit cases. This completes the proof of the proposition. \\
	\end{proof}
	At this point we obtained two equivalent definitions of modulation spaces of the type $M^{p,q}_s$. These are indeed helpful since both of them can be used to reveal different properties. 
	
	\section{Properties of Modulation Spaces}
		Strongly connected to modulation spaces is the theory of weighted mixed norm Lebesgue spaces which we already used in the proof of Proposition \ref{normequivalence} without actually defining them. Hence, we introduce these spaces in order to show properties of modulation spaces.
	\begin{df} 
		Let $1\leq p,q\leq \infty$ be the integrability parameters and $s,\sigma$ be real numbers. Then the weighted mixed-norm space of all \emph{Lebesgue measurable functions} $F$ on $\R^{2n}$ is denoted by $L^{p,q}_{s,\sigma}(\R^{2n})$ and consists of all $F$ such that its norm
		\[ \|F\|_{L^{p,q}_{s,\sigma}(\R^{2n})} = \left( \int_{\R^n} \left( \int_{\R^n} |F(x,\xi)\langle x\rangle^\sigma \langle \xi\rangle^s|^p dx \right)^{\frac{q}{p}} d\xi \right)^{\frac{1}{q}} \]
		is finite. 
	\end{df}
	\begin{rem}
		Analogously to the remark of Definition \ref{modCont}, all notations apply to weighted Lebesgue spaces $L^{p,q}_{s,\sigma}(\R^n)$. 
	\end{rem}
		Properties of the spaces $L^{p,q}_{s,\sigma}(\R^n)$ were investigated in \cite{benedek}. In fact it is a Banach space and we know the dual space for $p,q\in [1,\infty)$, namely the dual space $(L^{p,q}_{s,\sigma})^*$ of $L^{p,q}_{s,\sigma}$ is given by $(L^{p,q}_{s,\sigma})^*=L^{p',q'}_{-s,-\sigma}$, where $p'$ and $q'$ denote the corresponding conjugated exponents, respectively. From these observations we can deduce some fundamental properties of weighted modulation spaces which were already revealed in \cite{groechenig}. For a more comprehensive understanding of some of the subsequent results we can refer to \cite{reich}, where proofs or respectively their ideas are shown. \\
		First of all we state the direct connection between modulation spaces and Lebesgue spaces in the following proposition. 
		\begin{thm} \label{MpqLpq}
			Let $\phi\in \S(\R^n)$ be a fixed non-zero window function. Then the function $f$ belongs to the modulation space $\mathring{M}^{p,q}_{s,\sigma}(\R^n)$ if and only if $V_\phi f\in L^{p,q}_{s,\sigma}(\R^{2n})$ for $1\leq p,q\leq \infty$ and $s,\sigma\in\R$. Moreover, different window functions yield equivalent norms.
		\end{thm}
		Furthermore the subsequent results can be proved. 
		\begin{thm} \label{MpqBanach}
			Assume $s,\sigma\in\R$ and the integrability parameters $1\leq p,q\leq \infty$. Then the modulation space $\mathring{M}^{p,q}_{s,\sigma}(\R^n)$ is a Banach space.
		\end{thm}
		\begin{thm} \label{Mduality}
			Let $s,\sigma\in\R$ and $p,q\in [1,\infty)$ with $p',q'$ such that $\frac{1}{p}+\frac{1}{p'}=\frac{1}{q}+\frac{1}{q'}=1$. Then $(\mathring{M}^{p,q}_{s,\sigma})^*=\mathring{M}^{p',q'}_{-s,-\sigma}$, where the duality is given by
			\begin{equation} \label{MpqDuality}
				( f, g )_{L^2} = \iint_{\R^{2n}} V_\phi f(x,\xi) \overline{V_\phi g(x,\xi)} \,dx\, d\xi 
			\end{equation}
			for $f\in \mathring{M}^{p,q}_{s,\sigma}(\R^n)$ and $g\in \mathring{M}^{p',q'}_{-s,-\sigma}(\R^n)$.
		\end{thm}
		In Definition \ref{modCont} we defined scales of modulation spaces. We can describe their relations in more detail. 		
		\begin{prop} \label{ModInclusion}
			Let $p_1,p_2,q_1,q_2\in [1,\infty]$ such that $p_1\leq p_2$ and $q_1\leq q_2$. Additionally assume $s_1,s_2,\sigma_1,\sigma_2 \in \R$ to be real numbers, where $s_2\leq s_1$ and $\sigma_2\leq \sigma_1$. Then it holds:
			\begin{enumerate}
				\item the inclusions  $\S(\R^n) \subset \mathring{M}^{p_1,q_1}_{s_1,\sigma_1}(\R^n) \subset \mathring{M}^{p_2,q_2}_{s_2,\sigma_2}(\R^n) \subset \S'(\R^n)$;
				\item if $p, q \in [1,\infty)$ and $s,\sigma\in \R$, then $\S(\R^n)$ is dense in $\mathring{M}^{p,q}_{s,\sigma}(\R^n)$;
				\item additionally to the assumptions above suppose $q_2<\infty$, then $\mathring{M}^{\infty, q_1}_s(\R^n)$ is dense in $\mathring{M}^{\infty, q_2}_s(\R^n)$. However the opposite statement is not true, i.e., taking $p_1 \leq p_2 < \infty$ we do not have a dense inclusion $\mathring{M}^{p_1, \infty}_s(\R^n) \subset \mathring{M}^{p_2,\infty}_s(\R^n)$. 
			\end{enumerate}
		\end{prop}
		\begin{proof}
			For the proofs of $1.$ and $2.$ see Theorem 12.2.2 and Proposition 11.3.4 in \cite{groechenig}. \\
			In order to prove $3.$ we mainly stress Gabor expansion which is explained in \cite{groechenig}. Assume $\phi$ and $\psi$ to be admissible window functions. Moreover, we introduce the following notation
		\[ \psi_{j,k} := M_{\xi_k} T_{x_j} \psi = e^{\i \xi_k x} \psi(x-x_j). \]
		Then Corollary 12.2.6 in \cite{groechenig} yields 
		\begin{eqnarray*}
			f & = & \sum_{j,k \in \Z^n} \langle f, M_{\xi_k} T_{x_j} \phi \rangle M_{\xi_k} T_{x_j} \psi \\
				& = &  \sum_{j,k \in \Z^n} (V_\phi f)(x_j,\xi_k) \psi_{j,k}
		\end{eqnarray*}
		with unconditional convergence for $p,q<\infty$ and weak$^*$ convergence in the limit case $p,q=\infty$. This result is sufficient for our considerations since we are only treating distributions anyway. By the same corollary we obtain equivalence of norms, i.e., there exists constants $C_1, C_2 >0$ such that
		\begin{equation} \label{gaborEquiv}
			C_1 \| f\|_{\mathring{M}^{p,q}_s} \leq \|(V_\phi f)(x_j, \xi_k) \|_{l^{p,q}_s} \leq C_2 \|f\|_{\mathring{M}^{p,q}_s}. 
		\end{equation}
		So let $q_1\leq q_2 <\infty$ and $f\in \mathring{M}^{\infty, q_2}$. For a number $N\in \N$ we define a function $f_N$ by
		\[ f_N = \sum_{k\in \Z^n} \sum_{j\in \Z^n} \chi_N(k) (V_\phi f)(x_j, \xi_k) \psi_{j,k}, \]
		where $\chi_N$ is the characteristic function with respect to $k$, i.e.,
		\[ \chi_N(k) = \begin{cases}
					1, & |k|\leq N \\
					0, & |k| > N.
				\end{cases} 
		\]
		Since the coefficient $\chi_N(k) (V_\phi f)(x_j, \xi_k)$ of $f_N$ can be considered as a finite sequence with respect to $k$ we know that $\{\chi_N(k) (V_\phi f)(x_j, \xi_k)\}_{k\in\Z^n} \in l^q_s$ for any $q \in [1, \infty)$, in particular for $q_1$. Hence, it also holds 
		\[ \{\chi_N(k) (V_\phi f)(x_j, \xi_k)\}_{j,k\in\Z^n} \in l^{\infty, q_1}_s. \]
		Using the previous considerations we compute 
		\begin{eqnarray*}
			\| f- f_N\|_{\mathring{M}^{\infty, q_2}_s} & \leq & C_1 \| (V_\phi f)(x_j,\xi_k) - \chi_N(k) (V_\phi f)(x_j, \xi_k) \|_{l^{\infty, q_2}_s} \\
				& = & \| \langle k \rangle^s \sup_{j\in \Z^n} | (V_\phi f)(x_j,\xi_k) - \chi_N(k) (V_\phi f)(x_j, \xi_k) | \|_{l^{q_2}} \\
				& = & \left( \sum_{k\in\Z^n} \langle k \rangle^{q_2} \sup_{j\in \Z^n} | (V_\phi f)(x_j,\xi_k) - \chi_N(k) (V_\phi f)(x_j, \xi_k) |^{q_2} \right)^{\frac{1}{q_2}} \\
				& \to & 0 
		\end{eqnarray*}
		as $N\to \infty$. Hence, the first statement is proved. \\
		The inclusion $\mathring{M}^{p_1, \infty}_s (\R^n) \subset \mathring{M}^{p_2,\infty}_s (\R^n)$ follows by Theorem 12.2.2 in \cite{groechenig}. Using the same techniques as in the first part of the proof we obtain the sequence of coefficients $\{ \chi_N(j) (V_\phi f)(x_j, \xi_k)\}_{j,k\in\Z^n}$ of $f_N$ which is a finite sequence with respect to $j$ and therefore belongs to $l^{p_1,\infty}_s$. Defining now the function
		\[ f_N = \sum_{k\in \Z^n} \sum_{j\in \Z^n} \chi_N(j) (V_\phi f)(x_j, \xi_k) \psi_{j,k} \quad \in \mathring{M}^{p_1,\infty}_s \]
		yields the following computation
		\begin{eqnarray*}
			\| f- f_N\|_{\mathring{M}^{p_2, \infty}_s} & \leq & C_1 \| (V_\phi f)(x_j,\xi_k) - \chi_N(j) (V_\phi f)(x_j, \xi_k) \|_{l^{p_2, \infty}_s} \\
				& = & \sup_{k\in\Z^n} \left( \langle k \rangle^s \| (V_\phi f)(x_j,\xi_k) - \chi_N(j) (V_\phi f)(x_j, \xi_k) \|_{l^{p_2}} \right) \\
				& = & \sup_{k\in\Z^n} \Bigg( \langle k \rangle^s \Big( \sum_{j\in\Z^n} |(V_\phi f)(x_j,\xi_k) - \chi_N(j) (V_\phi f)(x_j, \xi_k)|^{p_2} \Big)^{\frac{1}{p_2}} \Bigg) \\
				& = & \sup_{k\in\Z^n} \Bigg( \langle k \rangle^s \Big( \sum_{\substack{j\in\Z^n, \\ |j|>N}} |(V_\phi f)(x_j,\xi_k)|^{p_2} \Big)^{\frac{1}{p_2}} \Bigg).
		\end{eqnarray*}
		Now we will establish a counterexample, i.e., we construct a sequence $\{c_{j,k}\}_{j,k\in\Z^n}$ satisfying
		\begin{equation} \label{normCondition} 
			\sup_{k\in\Z^n} \Bigg( \langle k \rangle^s \sum_{j\in\Z^n} |c_{j,k}|^{p_2} \Bigg)^{\frac{1}{p_2}} < \infty 
		\end{equation}
		but 
		\begin{equation} \label{counterexample}
			\sup_{k\in\Z^n} \Bigg( \langle k \rangle^s \sum_{\substack{j\in\Z^n, \\ |j|>N}} |c_{j,k}|^{p_2} \Bigg)^{\frac{1}{p_2}} 
		\end{equation}
		does not tend to zero as $N\to \infty$. Let $N_0$ be a fixed positive integer and $\varphi$ be a positive bounded function such that $\supp \varphi \subset \{ x\in \R^n : |x|\leq N_0\}$. Now set $c_{j,k} = \varphi(j-k)$. Obviously (\ref{normCondition}) is fulfilled but we see that for every $N$ we can choose $k$ such that $N\leq |k|\leq N+N_0$. Thus, the term (\ref{counterexample}) never gets arbitrarily small. Hence $\mathring{M}^{p_1, \infty}_s (\R^n)$ is not dense in $\mathring{M}^{p_2, \infty}_s (\R^n)$ which completes the proof.
		\end{proof}
		\begin{rem}
			The Gabor analysis of modulation spaces as shown in \cite{groechenig} is a very strong tool. And at this point we only use it as a tool in order to prove some basic results about characterizations of modulation spaces. We do not present any theory of Gabor analysis. \\
			Note that an immediate consequence of the latter proposition is that the standard modulation space increases with its integrability parameters $p$ and $q$. 
		\end{rem}
		As used in the proof of Proposition \ref{ModInclusion} we have an appropriate characterization of modulation space functions in terms of Gabor expansion. Note that in general we consider elements of modulation spaces as distributions. Now we want to deduce a characterization result for periodic modulation space functions. Recall that admissible periodic functions can be represented by their corresponding Fourier series. It is shown in \cite{hormander} that we can find a representation of periodic distributions in terms of a sum namely Poisson's summation formula for distributions. Let $u\in \S'(\R^n)$ be periodic such that 
	\[ u(x-\alpha) = u(x), \qquad \alpha\in\Z^n \]
	and suppose $\phi\in\S(\R^n)$ to be a function satisfying
	\[ \sum_{\alpha\in\Z^n} \phi(x-\alpha) = 1. \]
	It is known that Poisson's summation formula holds for Schwartz functions. This and Fourier's inversion formula yield 
	\[ u = \sum_{\alpha\in\Z^n} c_\alpha e^{2\pi \i x \cdot \alpha} , \]
	where the coefficients are defined by
	\begin{equation} \label{coefficients}
		c_\alpha = ( u(x), \phi(x) e^{-2\pi\i x\cdot \alpha } )_{L^2(\R^n)}.
	\end{equation}
	\begin{prop}
		Suppose $1\leq q<\infty$ and $s>1$. If $f\in\S'(\R^n)$ is a periodic tempered distribution with period $2\pi$ in each variable, then
		\[ f = \sum_{\alpha\in\Z^n} c_\alpha e^{\i x\cdot \alpha} \]
		with coefficients
		\[ c_\alpha = ( f(x), \phi(x) e^{-\i x\cdot \alpha } )_{L^2(\R^n)} \]
		similar to \eqref{coefficients}. \\
		Furthermore $f\in \mathring{M}^{\infty,q}_s(\R^n)$ if and only if $\{c_{\alpha}\}_{\alpha\in\Z^n} \in l^q_s(\Z^n)$.
	\end{prop}
	\begin{proof}
		By \cite{hormander} we know that Poisson's summation formula holds for tempered distributions $\S'$, where the coefficients are defined as above. \\
		So the first step is to show $f\in \mathring{M}^{\infty,q}_s$ implies $\{c_{\alpha}\}_{\alpha\in\Z^n} \in l^q_s$. Let $f\in \mathring{M}^{\infty, q}_s (\R^n)$ and $\phi \in\S(\R^n)$. By the mean value theorem and (\ref{gaborEquiv}) it follows
		\begin{eqnarray*}
			\|\{c_\alpha\}_\alpha \|_{l^q_s} & = & \left\| \langle \alpha \rangle^s \int_{\R^n} f(x) \phi(x) e^{-\i x \cdot \alpha} \, dx \right\|_{l^q} \\
				& \leq & \left\| \langle \alpha \rangle^s \sup_{\beta\in\Z^n} \left| \int_{\R^n} f(x) \phi(x-x_\beta) e^{- \i x\cdot \xi_\alpha} \, dx \right| \, \right\|_{l^q} \\
				& \leq & C_1 \| (V_\phi f) (x_\beta,\xi_\alpha) \|_{l^{\infty, q}_s} \\
				& \leq & C_2 \|f\|_{\mathring{M}^{\infty,q}_s}.
		\end{eqnarray*}
		Now suppose $\{c_\alpha\}_{\alpha\in\Z^n} \in l^q_s(\Z^n)$. First note that 
		\begin{eqnarray*}
			V_\phi \Big(\sum_{\alpha\in\Z^n} c_\alpha e^{ \i x\cdot \alpha} \Big)(x,\xi) & = & \sum_{\alpha\in\Z^n} c_\alpha V_\phi (e^{\i x \cdot \alpha})(x,\xi) \\
				& = & \sum_{\alpha\in\Z^n} c_\alpha \int_{\R^n} e^{ \i s \cdot \alpha} \overline{\phi(s-x)} e^{-\i s\cdot \xi} \, ds \\
				& \stackrel{[y=s-x]}{=} & \sum_{\alpha\in\Z^n} c_\alpha \int_{\R^n} e^{\i x\cdot \alpha} e^{\i y\cdot \alpha} \overline{\phi(y)} e^{-\i x\cdot \xi} e^{-\i y\cdot \xi} \, dy \\
				& = & \sum_{\alpha\in\Z^n} c_\alpha e^{-\i x\cdot (\xi-\alpha)} \int_{\R^n} \overline{\phi(y)} e^{-\i y\cdot(\xi-\alpha)} \, dy \\
				& = & \sum_{\alpha\in\Z^n} c_\alpha e^{-\i x\cdot (\xi-\alpha)} \hat{\bar{\phi}}(\xi-\alpha). 
		\end{eqnarray*}
		Thus, we obtain
		\begin{eqnarray*}
			\| \sum_{\alpha\in\Z^n} c_\alpha e^{ \i x\cdot \alpha} \|_{\mathring{M}^{\infty, q}_s} & = & \| \langle \xi \rangle^s \sum_{\alpha\in\Z^n} c_\alpha \hat{\bar{\phi}} (\xi-\alpha) e^{-\i x\cdot(\xi-\alpha)} \|_{L^{\infty,q}} \\
				& = & \| \langle \cdot \rangle^s \sum_{\alpha\in\Z^n} c_\alpha \hat{\bar{\phi}} (\cdot-\alpha) \|_{L^{q}} \\
				& \leq & \| \langle \cdot \rangle^s \sum_{\alpha\in\Z^n} |c_\alpha| |\hat{\phi} (\cdot-\alpha)| \|_{L^{q}} \\
				& = &  \left\| \langle \cdot \rangle^s \sum_{\alpha\in\Z^n} \left(|c_\alpha| |\hat{\phi}(\cdot-\alpha)|^{\frac{1}{q}} \right) \left( |\hat{\phi}(\cdot-\alpha)|^{\frac{1}{q'}} \right) \right\|_{L^{q}} \\
				& \leq & \left\| \langle \cdot \rangle^s \left( \sum_{\alpha\in\Z^n} |c_\alpha|^q |\hat{\phi}(\cdot-\alpha)| \right)^{\frac{1}{q}} \left( \sum_{\alpha\in\Z^n} |\hat{\phi}(\cdot-\alpha)| \right)^{\frac{1}{q'}} \right\|_{L^{q}} \\
				& \leq & C_1  \left\| \langle \cdot \rangle^s \left( \sum_{\alpha\in\Z^n} |c_\alpha|^q |\hat{\phi}(\cdot-\alpha)| \right)^{\frac{1}{q}} \right\|_{L^{q}} \\
				& = & C_1 \left( \int_{\R^n} \langle \xi \rangle^{sq} \sum_{\alpha\in\Z^n} |c_\alpha|^q |\hat{\phi}(\xi-\alpha)| \, d\xi \right)^{\frac{1}{q}} \\
				& \leq & C_1 \left( \int_{\R^n} \sum_{\alpha\in\Z^n} \langle \alpha \rangle^{sq} |c_\alpha|^q |\hat{\phi}(\xi-\alpha)| \, d\xi \right)^{\frac{1}{q}} \\
				& = & C_1 \left( \sum_{\alpha\in\Z^n} \langle \alpha \rangle^{sq} |c_\alpha|^q \int_{\R^n} |\hat{\phi}(\xi-\alpha)| \, d\xi \right)^{\frac{1}{q}} \\		
				& \leq & C_2 \|\{c_\alpha\}_\alpha \|_{l^q_s}.
		\end{eqnarray*}
	\end{proof}
	\begin{rem}
		We assumed the period of the periodic distribution to be $1$ or $2\pi$ in each variable, respectively. The stated result can be naturally generalized for a period $T=(T_1,T_2,\ldots,T_n) \in \Z^n$. 
	\end{rem}
		Summarizing we have stated some basic properties of modulation spaces and mentioned some sensible characterizations of functions in modulation spaces. \\ Since our goal is to apply the theory of modulation spaces to partial differential equations another essential investigation concerns the Fourier transform. 
		\begin{prop} \label{propFourier}
			The set of all Fourier transforms on $\mathring{M}^{p,q}_{s,\sigma}$ is equal to the modulation space $W^{q,p}_{\sigma,s}$ for $1\leq p,q \leq \infty$ and $s,\sigma\in\R$. 
		\end{prop}
		\begin{rem}
			The proof of this proposition also yields 
			\[ \hat{f}\in \mathring{M}^{p,q}_{s,\sigma}(\R^n) \iff f\in W^{q,p}_{\sigma,s}(\R^n). \]
		\end{rem}
		\begin{prop} \label{coincideMW}
			The spaces $\mathring{M}^{p,q}_{s,\sigma}$ and $W^{q,p}_{s,\sigma}$ coincide if $p=q$.
		\end{prop}
		\begin{prop} \label{inclMW}
			Let $s,\sigma\in\R$. For $p\leq q$ it holds $W^{p,q}_{s,\sigma}\subseteq \mathring{M}^{p,q}_{s,\sigma}$. Analogously, $q\leq p$ gives $\mathring{M}^{p,q}_{s,\sigma}\subseteq W^{p,q}_{s,\sigma}$.
		\end{prop}
		A multiplication result completes the fundamental preparations for eventual investigations of partial differential equations. The following theorem was originally shown by Feichtinger.
		\begin{thm} \label{multiplication}
			Let $p_j,q_j\in [1,\infty]$, where $j=0,1,2$, such that 
			\[ \frac{1}{p_1} + \frac{1}{p_2} = 1 + \frac{1}{p_0} \quad \mbox{and} \quad \frac{1}{q_1} + \frac{1}{q_2} = \frac{1}{q_0}. \]
			Furthermore, let $s,\sigma_0,\sigma_1,\sigma_2 \in \R$ be numbers such that $\sigma_1+\sigma_2=\sigma_0$. Then it holds 
			\begin{equation} \label{sub1}
				W^{q_1,p_1}_{s,\sigma_1} \cdot W^{q_2,p_2}_{s,\sigma_2} \subseteq W^{q_0,p_0}_{s,\sigma_0}.
			\end{equation}
			Now assume that
			\[ \frac{1}{p_1} + \frac{1}{p_2} = \frac{1}{p_0} \quad \mbox{and} \quad \frac{1}{q_1} + \frac{1}{q_2} = 1 + \frac{1}{q_0}. \]
			Then
			\begin{equation} \label{sub2}
				\mathring{M}^{p_1,q_1}_{s,\sigma_1} \cdot \mathring{M}^{p_2,q_2}_{s,\sigma_2} \subseteq \mathring{M}^{p_0,q_0}_{s,\sigma_0}
			\end{equation}
			and
			\begin{equation} \label{sub3}
				W^{p_1,q_1}_{s,\sigma_1} \cdot W^{p_2,q_2}_{s,\sigma_2} \subseteq W^{p_0,q_0}_{s,\sigma_0}.
			\end{equation}
		\end{thm}
		\begin{proof}
			Cf. Theorem 2.4 in \cite{toft}.
		\end{proof}
	
\section{Gevrey-modulation spaces} \label{SecGMspace}

	\subsection{Introduction}
	We have got a basic understanding of weighted modulation spaces. We also introduced two equivalent definitions of modulation spaces. In all previous proofs we used the continuous Definition \ref{STFT} of the short-time Fourier transform. So we have not made use of Definition \ref{defdecomp} of modulation spaces yet. The advantages of the frequency-uniform decomposition appear in the proof of Theorem \ref{algebra}, Lemma \ref{estSuper} and Theorem \ref{Superposition}. Another important aspect which we have not considered yet is the meaning of the weight functions. Note that subsequently the weight with respect to the $x$-variable will not be of interest since the weight function with respect to the $\xi$-variable corresponds to regularity properties. Now we will focus on growth properties of those weights. So far we worked in the Sobolev frame, i.e., we considered weight functions $w(\xi) = \langle \xi \rangle^s$ which is rather reasonable. For instance suppose that $f\in\S(\R^n)$. Since we assumed the window function $\phi \in\S(\R^n)$ to be a Schwartz function and by Proposition \ref{STFTdecay} we know that for all $N\geq 0$ there exists a constant $C_N>0$ such that
	\[ |V_\phi f (x,\xi)| \leq C_N (1 + |x| + |\xi|)^{-N}. \]
	This statement also holds vice versa. \\
	In order to obtain better results we adjust the weight function $w=w(\xi)$. In particular we will subsequently work with a function $w(\xi) = e^{|\xi|^{\frac{1}{s}}}$, where $s>1$ is the so-called weight parameter. That basically means we are using weights of Gevrey type. The motivation to follow this strategy comes from \cite{brs}. There the authors used spaces of Gevrey type, i.e., function spaces defined by the behavior of the Fourier transform. In fact we are in a similar situation when treating modulation spaces. Since the future work is aimed at applying modulation spaces to partial differential equations we need to prepare some tools. We will prove an algebra result that can be used to include analytic non-linearities in partial differential equations. It also helps us to show a superposition result. This in turn provides us a generalization of our considerations to particular non-analytic non-linearities in partial differential equations. Since it is not in the least trivial to obtain superposition operators for spaces which are characterized by the Fourier transform side we need to choose admissible weights, i.e., weights of Gevrey type $w(\xi) = e^{|\xi|^{\frac{1}{s}}}$ with $s>1$. This implies another problem. The space of window functions for the short-time Fourier transform namely the Schwartz space does not sensibly define modulation spaces anymore. Therefore we introduce the \emph{Gelfand-Shilov space} as in \cite{tknn}.
	
	\subsection{On Gelfand-Shilov Spaces and a Definition of Gevrey-modulation Spaces}
	\begin{df} \label{gelfandshilov}
		Let $h>0$ and $s\in\R$ be fixed numbers. Then the Gelfand-Shilov space $\S_{s,h}(\R^n)$ consists of all functions $f\in C^\infty (\R^n)$ satisfying 
		\[ \sup_{x \in \R^n} |x^\alpha \partial^\beta f(x)| \leq C h^{|\alpha|+|\beta|} (\alpha!\beta!)^s \]
		for all multi-indices $\alpha, \beta \in \N^n$ with a positive constant $C$. Thus the Gelfand-Shilov space is the set
		\[ \S_{s,h}(\R^n) = \{ f\in C^\infty (\R^n) : \|f\|_{\S_{s,h}(\R^n)} < \infty \}, \]
		where the norm is defined as
		\[ \|f\|_{\S_{s,h}(\R^n)} = \sup_{x\in\R^n, \alpha,\beta\in\N^n} \frac{|x^\alpha \partial^\beta f(x)|}{h^{|\alpha|+|\beta|} (\alpha!\beta!)^s}. \]
	\end{df}
	Obviously, it holds $\S_{s,h}\subset \S$ and moreover, the Gelfand-Shilov spaces are increasing with the parameters $h$ and $s$. Due to \cite{gelfand} it is  well-known that these spaces are Banach spaces. \\
	We can also define the Gelfand-Shilov space $\S_s (\R^n)$ which is the inductive limit with respect to $\S_{s,h} (\R^n)$, i.e.,
	\[ \S_s(\R^n) = \bigcup_{h>0} \S_{s,h}(\R^n). \]
	In \cite{gelfand} it is shown that for $s<\frac{1}{2}$ the space $\S_s$ is trivial. So it is reasonable to consider only values $s\geq \frac{1}{2}$. By the argumentation in \cite{tknn} it also follows that the dual space $\S_{s,h}'$ of $S_{s,h}$ is a Banach space which contains the set of all tempered distributions $\S'$. Now we can define the so-called Gelfand-Shilov distribution space $\S_s'$ as the projective limit 
	\[ \S_s'(\R^n) = \bigcap_{h>0} \S_{s,h}'(\R^n). \]
	If we go back to Definition \ref{modCont} of the modulation space considering now weights of exponential type the short-time Fourier transform obviously needs a behavior similar to 
	\[ |V_\phi f(x,\xi)| \leq C e^{-\epsilon (|x|^{\frac{1}{s}} + |\xi|^{\frac{1}{s}})} \]
	for some $\epsilon>0$, where the strong decay with respect to the $\xi$-variable is necessary and the decay with respect to the $x$-variable is sufficient for the convergence of the integral.
	\begin{lma} \label{GSdef}
		The following conditions are equivalent: 
		\begin{enumerate}
			\item It holds $f\in \S_s(\R^n)$. 
			\item There is a constant $C>0$ and a number $h>0$ such that
				\begin{equation} \label{GSdefChar1}
					\sup_{x\in\R^n} |x^\alpha f(x)| \leq C h^{|\alpha|} (\alpha !)^s \mbox{ and } \quad \sup_{x \in \R^n} |\partial^\beta f(x)| \leq C h^{|\beta|} (\beta !)^s 
				\end{equation}
				for all multi-indices $\alpha, \beta \in \N^n$.
			\item There is a constant $C>0$ and a number $h>0$ such that
				\begin{equation} \label{GSdefChar2}
					\sup_{x\in\R^n} |x^\alpha f(x)| \leq C h^{|\alpha|} (\alpha !)^s \mbox{ and } \quad \sup_{\xi \in \R^n} |\xi^\beta \hat{f}(\xi)| \leq C h^{|\beta|} (\beta !)^s 
				\end{equation}
				for all multi-indices $\alpha, \beta \in \N^n$.
		\end{enumerate}
	\end{lma}
	\begin{proof}
		We basically use a simpler version of the proof of Theorem 2.3 in \cite{chung}. \\
		Assume that $f\in \S_s(\R^n)$. Then the inequalities (\ref{GSdefChar1}) and (\ref{GSdefChar2}) follow immediately by Definition \ref{gelfandshilov} since the space $\S_s$ is invariant under Fourier transform as stated in \cite{tknn}. \\
	Now let us assume that the inequalities (\ref{GSdefChar1}) hold. From the estimate
	\[ \left( \begin{array}{c} \alpha+\beta \\ \beta \end{array} \right) \stackrel{[\gamma=\alpha+\beta]}{=} \left( \begin{array}{c} \gamma \\ \gamma-\alpha \end{array} \right) = \left( \begin{array}{c} \gamma \\ \alpha \end{array} \right) \leq \sum_{\alpha\leq \gamma} \left( \begin{array}{c} \gamma \\ \alpha \end{array} \right) =  2^{|\gamma|} =  2^{|\alpha+\beta|} \]
	we can deduce
	\begin{eqnarray*}
		((\alpha +\beta)!)^s & = & (\alpha!)^s (\beta!)^s \left( \begin{array}{c} \alpha+\beta \\ \beta \end{array} \right)^s \\
			& \leq & 2^{s (|\alpha|+|\beta|)} (\alpha!)^s (\beta!)^s. 
	\end{eqnarray*}
	Hence we can say that there exists a constant $H>0$ such that
	\begin{equation} \label{abest}
		((\alpha+\beta)!)^s \leq H^{|\alpha| + |\beta|} (\alpha!)^s (\beta!)^s.
	\end{equation}
	Later on we want to estimate the $L^2$-norm of the term $x^\alpha \partial^\beta f(x)$. Therefore we need to show that there exists a constant $C>0$ such that 
	\[ \| x^\alpha \partial^\beta f(x)\|_{L^\infty (\R^n)} \leq C \|x^\alpha \partial^\beta f(x)\|_{L^2(\R^n)}. \]
	For simplicity we only consider the case $n=1$. For higher dimensions the argumentation stays the same. Note that $g(x) = \int_{-\infty}^x g'(x) \, dx$ and $(x^\alpha \partial^\beta f(x))' = \alpha x^{\alpha-1} \partial^\beta f(x) + x^\alpha \partial^{\beta+1} f(x)$. Then it holds
	\begin{eqnarray*}
		\|x^\alpha \partial^\beta f(x)\|_{L^\infty} & \leq & \|(x^\alpha \partial^\beta f(x))' \|_{L^1} \\
			& \leq & \|(1+x^2) (x^\alpha \partial^\beta f(x))' \|_{L^2} \| (1+x^2)^{-1}\|_{L^2} \\
			& \leq & C_0 \| (1+x^2) (\alpha x^{\alpha-1} \partial^\beta f(x) + x^\alpha \partial^{\beta+1} f(x))\|_{L^2} \\
			& \leq & C_0 \Big( \alpha \|x^{\alpha-1} \partial^\beta f(x)\|_{L^2} + \| x^\alpha \partial^{\beta+1} f(x) \|_{L^2} \\
			& & \qquad \qquad + \alpha \|x^{\alpha +1} \partial^\beta f(x)\|_{L^2} + \| x^{\alpha+2} \partial^{\beta+1} f(x)\|_{L^2} \Big),
	\end{eqnarray*}
	where we used Cauchy-Schwarz and Minkowski inequality. Since $f\in \S_s(\R^n)$ and our estimates hold for all multi-indices $\alpha, \beta$ we can simply write
	\[ \|x^\alpha \partial^\beta f \|_{L^\infty (\R^n)} \leq C \| x^\alpha \partial^\beta f\|_{L^2(\R^n)}. \]
	Thus, we can estimate the $L^2$-norm instead of $L^\infty$-norm. By applying partial integration, the Leibniz rule and the Cauchy-Schwarz inequality we get
		\begin{eqnarray*}
			\|x^\alpha \partial^\beta f(x)\|^2_{L^2(\R^n)} & = & \int_{\R^n} \left[ x^{2\alpha} \partial^\beta f(x) \right] \partial^\beta f(x) \, dx \\
				& = & (-1)^{|\beta|} \int_{\R^n} \partial^\beta \left( x^{2\alpha} \partial^\beta f(x) \right) f(x) \, dx \\
				& = & (-1)^{|\beta|} \int_{\R^n} \sum_{\gamma\leq \beta} \binom{\beta}{\gamma} \left(\partial^\gamma x^{2\alpha}\right) \partial^{\beta-\gamma} \left(\partial^\beta f(x) \right) f(x) \, dx \\
				& = & (-1)^{|\beta|} \int_{\R^n} \sum_{\gamma\leq \beta} \binom{\beta}{\gamma} \binom{2\alpha}{\gamma} \gamma! \partial^{2\beta -\gamma} f(x) x^{2\alpha -\gamma} f(x)\, dx \\
				& \leq & \sum_{\gamma\leq \beta} \binom{\beta}{\gamma} \binom{2\alpha}{\gamma} \gamma! \| \partial^{2\beta -\gamma} f(\cdot) \|_{L^2(\R^n)} \| x^{2\alpha -\gamma} f(\cdot) \|_{L^2(\R^n)} \\
				& \leq & \sum_{\gamma\leq \beta} \binom{\beta}{\gamma} \binom{2\alpha}{\gamma} \gamma! C h^{|2\beta-\gamma|} ((2\beta-\gamma)!)^s C h^{|2\alpha-\gamma|} ((2\alpha -\gamma)!)^s \\
				& \leq & C^2 h^{2 (|\alpha|+|\beta|)} ((2\beta)!)^s ((2\alpha)!)^s \sum_{\gamma\leq \beta} \frac{\binom{\beta}{\gamma} \binom{2\alpha}{\gamma} \gamma!}{(\gamma !)^s (\gamma !)^s} \\
				& \leq & C^2 h^{2 (|\alpha|+|\beta|)} ((2\beta)!)^s ((2\alpha)!)^s \sum_{\gamma} \binom{\beta}{\gamma} \sum_{\gamma} \binom{2\alpha}{\gamma} \sum_{\gamma} \frac{1}{(\gamma!)^s} \\
				& \leq & C_0^2 h^{2 (|\alpha|+|\beta|)} ((2\beta)!)^s ((2\alpha)!)^s 2^{|\beta|} 2^{2|\alpha|} \\
				& \leq & C_0^2 (2h)^{2 (|\alpha|+|\beta|)} ((2\beta)!)^s ((2\alpha)!)^s \\
				& \stackrel{(\ref{abest})}{\leq} & C_0^2 (2Hh)^{2 (|\alpha|+|\beta|)} (\beta!)^{2s} (\alpha!)^{2s}.
		\end{eqnarray*}
		Thus we have $f\in \S_s(\R^n)$. \\
		Now suppose that (\ref{GSdefChar2}) holds. Then we only need to show that 
		\[ \sup_{x\in \R^n} |\partial^\beta f(x)| \leq C h^{|\beta|} (\beta !)^s. \]
		By Definition 1.1 in \cite{chungEquiv} we deduce that there exists the so-called associated function $N=N(\xi)$ of the sequence $N_\beta = (\beta !)^s$ on $[0,\infty)$ such that 
		\[ N(\xi) = \sup_{\beta \in \N^n} \log \frac{\xi^\beta}{(\beta !)^s}. \]
		Due to the theory presented in \cite{chung} we can even state the inverse result. In particular we have 
		\[ (\beta !)^{2s} = ((\beta-1) !)^s \beta^s (\beta !)^s \leq ((\beta-1) !)^s (\beta+1)^s (\beta !)^s = ((\beta-1) !)^s ((\beta+1) !)^s \]
		and by Definition 3.2 together with Proposition 3.4 in \cite{chung} it follows 
		\begin{equation} \label{defSeq}
			(\beta !)^s = \sup_{\xi \in \R^n} \frac{\xi^\beta}{e^{N(\xi)}},
		\end{equation}
		where $N(\xi)$ is the associated function of $(\beta !)^s$. Now (\ref{GSdefChar2}) and the previous considerations yield
		\begin{eqnarray*}
			|\hat{f}(\xi)| & \leq & C \inf_{\beta \in \N^n} \frac{h^{|\beta|} (\beta !)^s}{|\xi|^{|\beta|}} \\
				& = & C e^{-N(\frac{|\xi|}{h})}.				
		\end{eqnarray*}
		Moreover, using (\ref{defSeq}) the estimates
		\begin{eqnarray*}
			|\partial^\beta f(x)| & \leq & (2\pi)^{-n} \int_{\R^n} | e^{\i x\cdot \xi} \xi^\beta \hat{f}(\xi) | \, d\xi \\
				& \leq & C_1 \int_{\R^n} |\xi|^{|\beta|} e^{-N(\frac{|\xi|}{h})} \, d\xi \\
				& \leq & C_1 \sup_{\xi\in\R^n} \left[ |\xi|^{|\beta|} e^{-\frac{1}{2} N(\frac{|\xi|}{h})} \right] \int_{\R^n} e^{-\frac{1}{2} N(\frac{|\xi|}{h})} \, d\xi \\
				& \leq & C_2 \sup_{\xi\in\R^n} \left( |\xi|^{2|\beta|} e^{-N(\frac{|\xi|}{h})} \right)^{\frac{1}{2}} \\
				& \stackrel{[\eta=\frac{\xi}{h}]}{=} & C_2 \sup_{\eta\in\R^n} \left( |\eta|^{2|\beta|} h^{2|\beta|} e^{-N(|\eta|)} \right)^{\frac{1}{2}} \\
				& = & C_2 \left( h^{2|\beta|} ((2\beta)!)^s \right)^{\frac{1}{2}} \\
				& \leq & C_2 h^{|\beta|} \sqrt{((2\beta)!)^s} \\
				& \stackrel{(\ref{abest})}{\leq} & C_2 h^{|\beta|} \sqrt{H^{2|\beta|} (\beta!)^{2s}} \\
				& = & C_2 (Hh)^{|\beta|} (\beta!)^s
		\end{eqnarray*}
		complete the proof. 		
	\end{proof}
	
	\begin{prop} \label{GScharact}
		Let $s\geq \frac{1}{2}$. Then there exist positive constants $C$ and $\epsilon$ such that a function $f$ can be characterized by 
		\begin{equation} \label{GSfunc}
			|f(x)| \leq C e^{-\epsilon |x|^{\frac{1}{s}}} \mbox{ and} \quad |\hat{f}(\xi)| \leq Ce^{-\epsilon |\xi|^{\frac{1}{s}}} 
		\end{equation}
		if and only if $f\in \S_s(\R^n)$. 
	\end{prop}
	\begin{proof}
		By Definition \ref{gelfandshilov} we can also write
		\[ |x|^{|\alpha|} |\partial^\beta f(x)| \leq C h^{|\alpha|+|\beta|} (\alpha !)^s (\beta !)^s. \]
		Thus, we get
		\[ \left( \frac{ \left( \frac{|x|^{\frac{1}{s}}}{(2h)^{\frac{1}{s}}} \right)^{|\alpha|}}{(\alpha !)} \right)^s |\partial^\beta f(x)| \leq C h^{|\beta|} (\beta !)^s \frac{1}{2^{|\alpha|}} \]
		and 
		\[ \frac{ \left( \frac{|x|^{\frac{1}{s}}}{(2h)^{\frac{1}{s}}} \right)^{|\alpha|}}{(\alpha !)} |\partial^\beta f(x)|^{\frac{1}{s}} \leq C^{\frac{1}{s}} h^{\frac{1}{s}|\beta|} (\beta !) \frac{1}{(2^{\frac{1}{s}})^{|\alpha|}}, \]
		respectively. 
		Taking now the sum over $\alpha$ on both sides we obtain
		\[ \sum_{\substack{\alpha \in \N^n, \\ |\alpha|\geq 0}} \frac{ \left( \frac{|x|^{\frac{1}{s}}}{(2h)^{\frac{1}{s}}} \right)^{|\alpha|}}{(\alpha !)} |\partial^\beta f(x)|^{\frac{1}{s}} \leq C^{\frac{1}{s}} h^{\frac{1}{s}|\beta|} (\beta !) \sum_{\substack{\alpha \in \N^n, \\ |\alpha|\geq 0}} \frac{1}{(2^{\frac{1}{s}})^{|\alpha|}}, \]
		that is by Taylor's formula,
		\[ e^{(\frac{1}{2h})^{\frac{1}{s}} |x|^{\frac{1}{s}}} |\partial^\beta f(x)|^{\frac{1}{s}} \leq C^{\frac{1}{s}} h^{\frac{1}{s}|\beta|} (\beta !) \big( 2^{\frac{n}{s}} \big). \]
		For $\beta=0$ and $\epsilon := \frac{s}{(2h)^{\frac{1}{s}}}$ follows the desired result
		\[ |f(x)| \leq C_1 e^{-\epsilon |x|^{\frac{1}{s}}}. \]
		By Lemma \ref{GSdef} we analogously get 
		\[ |\hat{f}(\xi) | \leq C_2 e^{-\epsilon |\xi|^{\frac{1}{s}}}. \]
	\end{proof}
	After getting general characterizations of functions from the Gelfand-Shilov space $\S_s$ we can find our desired characterization for the short-time Fourier transform of those functions.
	\begin{prop}
		Let $\phi \in \S_s(\R^n)$ be a fixed window function and $f\in \S_s'(\R^n)$. Then the following statements are equivalent:
		\begin{enumerate}
			\item $f\in \S_s(\R^n)$;
			\item $V_\phi f \in \S_s(\R^n)$;
			\item there exists a constant $C>0$ such that
				\[ |V_\phi f (x,\xi)| \leq C e^{-\epsilon(|x|^{\frac{1}{s}}+|\xi|^{\frac{1}{s}})}, \qquad (x,\xi)\in \R^{2n} \]
				for some $\epsilon>0$. 
		\end{enumerate}
	\end{prop}
	\begin{proof}
		Cf. Proposition 3.12 in \cite{groezimm}. At this point we only want to show that if $f\in S_s(\R^n)$, then we can find the characterization 
		\[ |V_\phi f (x,\xi)| \leq C e^{-\epsilon(|x|^{\frac{1}{s}}+|\xi|^{\frac{1}{s}})}. \]
		By already taking into account Lemma \ref{weightEstimate} and Proposition \ref{GScharact} we have 
		\begin{eqnarray*}
			|V_\phi f(x,\xi) | & = & (2\pi)^{-\frac{n}{2}} \left| \int_{\R^n} f(y) \overline{\phi(y-x)} e^{-\i y\cdot \xi} \, dy \right| \\
				& \leq & (2\pi)^{-\frac{n}{2}} \int_{\R^n} |f(y)| |\overline{\phi(y-x)}| |e^{-\i y\cdot \xi}| \, dy \\
				& \leq & C (2\pi)^{-\frac{n}{2}} \int_{\R^n} e^{-\epsilon |y|^{\frac{1}{s}}} e^{-\epsilon |y-x|^{\frac{1}{s}}} \, dy \\
				& \leq & C (2\pi)^{-\frac{n}{2}} \int_{\R^n} e^{-\epsilon |y|^{\frac{1}{s}}} e^{\epsilon |y|^{\frac{1}{s}}} e^{-\epsilon |x|^{\frac{1}{s}}} e^{-\epsilon \delta \min (|y-x|, |y|)^{\frac{1}{s}}} \, dy \\
				& = & C (2\pi)^{-\frac{n}{2}} e^{-\epsilon |x|^{\frac{1}{s}}} \int_{\R^n} e^{-\epsilon \delta \min (|y-x|, |y|)^{\frac{1}{s}}} \, dy \\
				& \leq & C_1 e^{-\epsilon |x|^{\frac{1}{s}}}.
		\end{eqnarray*}
		Stressing additionally Lemma \ref{STFTrepresent} it follows 
		\begin{eqnarray*}
			|V_\phi f(x,\xi) | & = & | e^{-\i x\cdot \xi}| |V_{\hat{\phi}} \hat{f} (\xi, -x)| \\
				& = & (2\pi)^{-\frac{n}{2}} \left| \int_{\R^n} \hat{f}(\eta) \overline{\hat{\phi}(\eta-\xi)} e^{\i \eta \cdot x} \, d\eta \right| \\
				& \leq & (2\pi)^{-\frac{n}{2}} \int_{\R^n}  |\hat{f}(\eta)| |\overline{\hat{\phi}(\eta-\xi)}| \, d\eta \\
				& \leq & C (2\pi)^{-\frac{n}{2}} \int_{\R^n} e^{-\epsilon |\eta|^{\frac{1}{s}}} e^{-\epsilon |\eta-\xi|^{\frac{1}{s}}} \, d\eta \\
				& \leq & C (2\pi)^{-\frac{n}{2}} \int_{\R^n} e^{-\epsilon |\eta|^{\frac{1}{s}}} e^{\epsilon |\eta|^{\frac{1}{s}}} e^{-\epsilon |\xi|^{\frac{1}{s}}} e^{-\epsilon \delta \min (|\eta-\xi|, |\eta|)^{\frac{1}{s}}} \, d\eta \\
				& = & C (2\pi)^{-\frac{n}{2}} e^{-\epsilon |\xi|^{\frac{1}{s}}} \int_{\R^n} e^{-\epsilon \delta \min (|\eta-\xi|, |\eta|)^{\frac{1}{s}}} \, d\eta \\
				& \leq & C_1 e^{-\epsilon |\xi|^{\frac{1}{s}}}.
		\end{eqnarray*}
		Putting these two results together we obtain
		\begin{eqnarray*}
			\frac{1}{2} (|V_\phi f(x,\xi)| + |V_\phi f(x,\xi)|) & \leq & \frac{1}{2} C_1 ( e^{-\epsilon |x|^{\frac{1}{s}}} + e^{-\epsilon |\xi|^{\frac{1}{s}}} ) \\
				& \leq & C_2 e^{-\epsilon_0 (|x|^{\frac{1}{s}} + |\xi|^{\frac{1}{s}})},
		\end{eqnarray*}
		which completes the proof.
	\end{proof}
		
 Thus, we have shown that it is reasonable to use Gelfand-Shilov spaces and we define the so-called Gevrey-modulation space as follows:
	\begin{df} \label{modExpWeight}
		The integrability parameters are given by $1\leq p,q \leq \infty$. Let $\phi\in \S_s(\R^n)\setminus \{0\}$ be a fixed window and assume $s > 1$ to be the weight parameter. Then the \emph{Gevrey-modulation space} $\mathcal{GM}^{p,q}_{s}(\R^n)$ is the set
		\[ \mathcal{GM}^{p,q}_{s} (\R^n) := \{ f\in \S_s'(\R^n): \|f\|_{\mathcal{GM}^{p,q}_{s}(\R^n)} < \infty \}, \]
		where the norm is defined as
		\[ \|f\|_{\mathcal{GM}^{p,q}_{s}(\R^n)} = \left( \int_{\R^n} \left( \int_{\R^n} |V_\phi f(x,\xi) |^p dx \right)^{\frac{q}{p}} e^{q |\xi|^{\frac{1}{s}}} d\xi \right)^{\frac{1}{q}} \]
		with obvious modifications when $p=\infty$ and/or $q=\infty$.
	\end{df}
	\begin{rm}
		Naturally we can define the norm in terms of the frequency-uniform decomposition
		\[ \|f\|_{\mathcal{GM}^{p,q}_{s}(\R^n)} = \left( \sum_{k\in\Z^n} e^{q |k|^{\frac{1}{s}}} \|\Box_k f\|_{L^p}^q \right)^{\frac{1}{q}}. \]
		Note that for the frequency-uniform decomposition we also need to use a function $\rho\in \S_s(\R^n)\setminus \{0\}$, cf. Definition \ref{defdecomp}. \\
		These two norms of Gevrey-modulation spaces are equivalent due to analogous arguments as in the proof of Proposition \ref{normequivalence}. Remark that applying the mean value theorem yields different constants when turning the integral into a sum. 
	\end{rm}

	\subsection{Embedding Results}
		Before treating algebra and superposition problems we shortly formulate an embedding result which usually arises together with these problems. Remark that we do not care about optimality since we will get a sufficiently good statement for Gevrey-modulation spaces introduced in Definition \ref{modExpWeight}. 
	\begin{prop} \label{SobolevL}
		Assume $N$ to be an integer. Then for every $\epsilon>0$ it holds
		\[ H_{N+\epsilon}^\infty \subseteq L_N^\infty \subseteq H_{N-\epsilon}^\infty, \]
		where the space $L_N^\infty$ consists of all functions which have bounded derivatives up to order $N$, i.e., if $f\in L_N^\infty(\R^n)$, then $\partial^\alpha f \in L^\infty(\R^n)$ for $|\alpha|\leq N$.
	\end{prop}
	\begin{proof}
		Let $B_s^{p,q}$ be the \emph{Besov space} with Lebesgue exponents $p,q$ and Sobolev parameter $s$. For details see \cite{bergh}. By argumentations in \cite{triebel} and \cite{toftEmbedding} we obtain the inclusions
		\[ B_{s+\epsilon}^{p,\infty} \subseteq B_s^{p,1} \]
		and
		\[ B_s^{p,1} \subseteq H_s^p \subseteq B_s^{p,\infty}. \]
		The embedding $L_N^\infty \subseteq B_N^{\infty,\infty}$ is implicitly given in chapter 6 in \cite{bergh}, the other embedding is explicitly given in \cite{bergh}. Hence it follows
		\[ L_N^\infty \subseteq B_N^{\infty,\infty} \subseteq B_{N-\epsilon}^{\infty,1} \subseteq H_{N-\epsilon}^\infty. \]
		Analogously we get $H^\infty_{N+\epsilon} \subseteq L_N^\infty$. 
	\end{proof}
	
	\begin{prop} \label{SobolevM}
		Let $\tilde{s}\in\R$ and $1\leq p,q \leq \infty$. Define $\theta_1=\theta_1(p,q)$ and $\theta_2=\theta_2(p,q)$ as follows
		\begin{eqnarray*}
			\theta_1(p,q) = \max(0,q^{-1}-\min(p^{-1},p'^{-1})), \\
			\theta_2(p,q) = \min(0,q^{-1}-\max(p^{-1},p'^{-1})),
		\end{eqnarray*}
		where $p'$ is the conjugate exponent of $p$. Then 
		\[ H^p_{\tilde{s}+\mu n \theta_1(p,q)}(\R^n) \subseteq M^{p,q}_{\tilde{s}}(\R^n) \subseteq H^p_{\tilde{s}+\mu n \theta_2(p,q)}(\R^n) \]
		for $\mu>1$. 
	\end{prop}
	\begin{proof}
		Cf. (0.4) in \cite{toftWeight}.
	\end{proof}
	
	\begin{prop} \label{embedding}
		Let $1\leq p,q \leq \infty$ and $\tilde{s} >\mu n$ with $\mu>1$, then $M^{p,q}_{\tilde{s}} (\R^n) \subset L^\infty(\R^n)$.
	\end{prop}
	\begin{proof}
		Due to Proposition \ref{ModInclusion} it suffices to prove $M^\infty_{\tilde{s}} \subset L^\infty$. \\
		Now Proposition \ref{SobolevM} yields 
		\[ M^\infty_{\tilde{s}} (\R^n) \subset H^\infty_{{\tilde{s}}-\mu n}(\R^n). \]
		Note that for the integrability parameters it holds $p=q=\infty$. The inclusion 
		\[ H^\infty_{{\tilde{s}}-\mu n}(\R^n) \subset B^\infty_{{\tilde{s}}-\mu n}(\R^n) \]
		is given in \cite{toftEmbedding}. Since $\tilde{s}>\mu n$ it follows by Theorem 3.3.1 in \cite{sickel} that
		\[ B^\infty_{\tilde{s}-\mu n} (\R^n) \subset L^\infty(\R^n) \]
		which completes the proof.
	\end{proof}
	At this point we immediately see that the strong weights in Gevrey-modulation spaces improve the situation. 
	\begin{cor} 
		Let $1\leq p,q \leq \infty$ and $s>1$, then $\mathcal{GM}^{p,q}_s (\R^n) \subset L^\infty(\R^n)$.
	\end{cor}
	\begin{proof}
		Fix the weight parameter $\tilde{s}>\mu n$ as in Proposition \ref{embedding} and set $\tilde{w}(\xi) = \langle \xi \rangle^{\tilde{s}}$. Then $w(\xi) = e^{| \xi|^{\frac{1}{s}}} > \tilde{w}(\xi)$ for arbitrary numbers $s>1$ and large $\xi\in\R^n$. This gives
		\[ \mathcal{GM}^{p,q}_s \subset M^{p,q}_{\tilde{s}} \subset L^\infty. \]
	\end{proof}
	
	\subsection{Multiplication Algebras}
	The next step consists of proving an essential property for Gevrey-modulation spaces. In order to obtain superposition results we need the algebra property of those spaces. Iwabuchi already showed some similar results in \cite{iwabuchi}. But he imposed conditions on the integrability parameters. Since the introduced Gevrey-modulation spaces $\mathcal{GM}^{p,q}_{s}$ have a better behavior than the usual modulation spaces defined in Definition \ref{defdecomp} we can show the fundamental algebra property for $\mathcal{GM}^{p,q}_{s}$. \\
	First of all we need the following lemma which is stated in \cite{brs}. 
	\begin{lma} \label{weightEstimate}
		If $s>1$, then it holds
		\[ e^{|k|^{\frac{1}{s}}} \leq e^{|l|^{\frac{1}{s}}} e^{|l-k|^{\frac{1}{s}}} e^{-\delta \min(|l-k|, |l|)^{\frac{1}{s}}}, \]
		where $k,l\in \Z^n$ and $0<\delta<1$.
	\end{lma}
	We also stress a version of Nikol'skij's inequality. 
	\begin{lma} \label{nikolskij}
		Let $1\leq p \leq q \leq\infty$ and $f$ be an integrable function with $\supp \F f(\xi) \subset B_r(\R^n) :=\{\xi\in \R^n: |\xi|\leq r\}$, i.e., the compact Fourier support of $f$ is contained in a ball around the origin with radius $r>0$. Then it holds
		\[ \|f\|_{L^q} \leq C r^{n(\frac{1}{p}-\frac{1}{q})} \|f\|_{L^p} \]
		with a constant $C>0$.
	\end{lma}
	\begin{proof}
		The idea of this proof is given in \cite{triebel}.\\
		By the assumption we know that $\supp \F f(r^{-1}\xi) \subset B_1(\R^n) :=\{\xi\in \R^n: |\xi|\leq 1\}$. If we take a compactly supported function $\phi\in\S(\R^n)$ with $\F\phi = 1$ in the ball $B_r(\R^n)$, then we have $\F f = \F f \cdot \F\phi$ and, therefore, we can rewrite $f$ as follows
		\[ f(x) = \F^{-1} (\F f\cdot \F\phi)(x) = \int_{\R^n} f(y) \phi(x-y) \, dy. \]
		By H\"older's inequality we obtain
		\begin{eqnarray*}
			|f(x)| & \leq & \int_{\R^n} |f(y)| |\phi(x-y)| \, dy \\
				& = & \left(\int_{\R^n} |f(y)|^p \, dy\right)^{\frac{1}{p}} \left(\int_{\R^n} |\phi(x-y)|^{p'} \, dy \right)^{\frac{1}{p'}} \\
				& \leq & C \|f\|_{L^p},
		\end{eqnarray*}
		where $C>0$ is a constant and $\frac{1}{p} + \frac{1}{p'} =1$. If we take the supremum with respect to $x$ we get
		\begin{equation} \label{LpEstimate}
			\|f\|_{L^\infty} \leq C \|f\|_{L^p}.
		\end{equation}
		Now we have
		\begin{eqnarray*}
		\|f(r^{-1}x)\|_{L^q} & = & \left(\int_{\R^n} |f(r^{-1}x)|^q \, dx \right)^{\frac{1}{q}} \\
			& = & \left(\int_{\R^n} |f(r^{-1}x)|^{q-p} |f(r^{-1}x)|^p \, dx \right)^{\frac{1}{q}} \\
			& \leq & \sup_{x\in\R^n} |f(r^{-1}x)|^{1-\frac{p}{q}} \left(\int_{\R^n} |f(r^{-1}x)|^p \, dx\right)^{\frac{1}{q}} \\
			& \stackrel{(\ref{LpEstimate})}{\leq} & C \|f(r^{-1}x)\|_{L^p}^{1-\frac{p}{q}} \|f(r^{-1}x)\|_{L^p}^{\frac{p}{q}} \\
			& = & C \|f(r^{-1}x)\|_{L^p}.
		\end{eqnarray*}
		By substituting $r^{-1}x$ by $y$ we get
		\begin{eqnarray*}
			\|f(r^{-1}x)\|_{L^q} & = & \left(\int_{\R^n} |f(r^{-1}x)|^q \, dx \right)^{\frac{1}{q}} \\
				& = &  \left(\int_{\R^n} r^n |f(y)|^q \, dy \right)^{\frac{1}{q}} \\
				& = & r^{\frac{n}{q}} \|f\|_{L^q}.
		\end{eqnarray*}
		Thus,
		\[ \|f\|_{L^q} \leq C r^{n(\frac{1}{p}-\frac{1}{q})} \|f\|_{L^p} \]
		and the proof is completed.
	\end{proof}
	
	\begin{thm} \label{algebra}
		Let $1\leq p,q\leq \infty$ and $s>1$. Assume that $f,g\in\mathcal{GM}^{2p,q}_{s}(\R^n)$, then $f\cdot g\in\mathcal{GM}^{p,q}_{s}(\R^n)$ and it holds
		\[ \|fg\|_{\mathcal{GM}^{p,q}_{s}} \leq C \|f\|_{\mathcal{GM}^{2p,q}_{s}} \|g\|_{\mathcal{GM}^{2p,q}_{s}} \]
		with a positive constant $C$ which is only dependent on the choice of the frequency-uniform decomposition, the dimension $n$ and the parameters $s, p, q$. \\
		In particular the Gevrey-modulation space $\mathcal{GM}^{p,q}_{s}(\R^n)$ is an algebra under multiplication. 
	\end{thm}
	\begin{proof}
		We will prove the algebra property since it is of main interest in this work. But we easily obtain the more general result if we simply omit to apply Nikol'skij's inequality, i.e. Lemma \ref{nikolskij}, with respect to the $L^p$ norm. \\
		We know that $\supp \sigma_k \subset Q_k := \{\xi \in\R^n| -1\leq \xi_i-k_i\leq 1, \, i=1,\ldots,n\}$. Further on we introduce the notations $f_j (x)= \F^{-1} (\sigma_j \F f)(x)$ and $g_l (x)= \F^{-1} (\sigma_l \F g)(x)$ for $j,l\in\Z^n$. We can obviously rewrite the product $f\cdot g$ as 
		\[ f\cdot g = \sum_{j,l\in Z^n} f_j \cdot g_l. \]
		Now we determine the Fourier support of $f_j \cdot g_l$. We have $\supp \F (f_j g_l) = \supp (\F f_j \ast \F g_l)$ and from
		\[ (\F f_j \ast \F g_l)(\xi) = \int_{Q_j} \F f_j(\eta) \F g_l(\xi-\eta) \, d\eta \]
		the subsequent computations can be deduced. For $i=1,\ldots,n$ we have $-1\leq \eta_i-j_i\leq 1$ and $-1\leq \xi_i-\eta_i-l_i\leq 1$. It follows that 
		\[ j_i+l_i-2 \leq \eta_i+l_i-1 \leq \xi_i \leq \eta_i+l_i+1 \leq j_i+l_i+2 \]
		and 
		\[ \supp \F f_j \ast \F g_l \subset \{\xi\in \R^n|  j_i+l_i-2 \leq \xi_i \leq j_i+l_i+2, \, i=1,\ldots,n\}. \]
		Thus, we have 
		\[ \F^{-1} (\sigma_k \F (f_j \cdot g_l)) \equiv 0 \]
		if
		\[ (-1+k_i, k_i +1) \cap (-2+j_i+l_i, 2+j_i+l_i) = \emptyset \]
		 for $i=1,\ldots, n$. Therefore the term does not vanish if
		\begin{eqnarray*}
			& \Longleftrightarrow & k_i +1 > -2 + j_i+l_i \mbox{ and } k_i-1<2+j_i+l_i \\
			& \Longleftrightarrow & k_i > -3 + j_i+l_i \mbox{ and } k_i < 3 + j_i+l_i \\
			& \Longleftrightarrow & -3+j_i +l_i < k_i < 3+j_i+l_i. 
		\end{eqnarray*}
		So we obtain
		\begin{eqnarray*}
			\left\|\F^{-1} \big(\sigma_k \F (f\cdot g)\big) \right\|_{L^p} & = & \Bigg\| \F^{-1} \Big( \sigma_k \F (\sum_{j,l\in\Z^n} f_j\cdot g_l) \Big) \Bigg\|_{L^p} \\
				& = & \Bigg\|\F^{-1} \Big( \sigma_k \F \big(\sum_{\substack{j,l \in \Z^n, \\ k_i-3<j_i+l_i < k_i +3, \\ i=1,\ldots, n}} f_j g_l \big) \Big) \Bigg\|_{L^p} \\
				& \stackrel{[r=j+l]}{=} & \Bigg\|\sum_{l\in\Z^n} \F^{-1} \Big( \sigma_k \F \big( \sum_{\substack{r \in \Z^n, \\ k_i-3<r_i < k_i +3, \\ i=1,\ldots, n}} f_{r-l} g_l \big) \Big) \Bigg\|_{L^p} \\
				& \leq & \sum_{\substack{r \in \Z^n, \\ k_i-3<r_i < k_i +3, \\ i=1,\ldots, n}} \sum_{l\in\Z^n} \| \F^{-1} \big( \sigma_k \F ( f_{r-l} g_l) \big)\|_{L^p} \\
				& \stackrel{[t=r-k]}{=} & \sum_{\substack{t \in \Z^n, \\ -3<t_i < 3, \\ i=1,\ldots, n}} \sum_{l\in\Z^n} \| \F^{-1} \big( \sigma_k \F ( f_{t-(l-k)} g_l) \big)\|_{L^p}.
		\end{eqnarray*}
		These preparations yield the following norm estimates
		\begin{align*}
			\left( \sum_{k\in\Z^n} e^{|k|^{\frac{1}{s}}q} \|\F^{-1} \big(\sigma_k \F (f\cdot g) \big)\|_{L^p}^q \right)^{\frac{1}{q}} & & \hspace{7cm}
		\end{align*}
		\begin{eqnarray*}	
		& \leq & \Bigg( \sum_{k\in\Z^n} e^{|k|^{\frac{1}{s}}q} \Bigg[ \sum_{\substack{t \in \Z^n, \\ -3<t_i < 3, \\ i=1,\ldots, n}} \sum_{l\in\Z^n} \| \F^{-1} \big(\sigma_k \F ( f_{t-(l-k)} g_l) \big)\|_{L^p} \Bigg]^q \Bigg)^{\frac{1}{q}} \\
			& \leq & \sum_{\substack{t \in \Z^n, \\ -3<t_i < 3, \\ i=1,\ldots, n}} \left( \sum_{k\in\Z^n} e^{|k|^{\frac{1}{s}}q} \left[ \sum_{l\in\Z^n} \| \F^{-1} \big(\sigma_k \F ( f_{t-(l-k)} g_l) \big)\|_{L^p} \right]^q \right)^{\frac{1}{q}} \\
			& = & \sum_{\substack{t \in \Z^n, \\ -3<t_i < 3, \\ i=1,\ldots, n}} \left( \sum_{k\in\Z^n} e^{|k|^{\frac{1}{s}}q} \left[ \sum_{l\in\Z^n} \| \big(\F^{-1} \sigma_k \ast (f_{t-(l-k)} g_l) \big)(\cdot)\|_{L^p} \right]^q \right)^{\frac{1}{q}} \\
			& \leq & \sum_{\substack{t \in \Z^n, \\ -3<t_i < 3, \\ i=1,\ldots, n}} \left( \sum_{k\in\Z^n} e^{|k|^{\frac{1}{s}}q} \left[ \sum_{l\in\Z^n} \| \F^{-1} \sigma_k\|_{L^1} \|f_{t-(l-k)} g_l\|_{L^p} \right]^q \right)^{\frac{1}{q}} \\
			& \leq & \sup_{k\in\Z^n} \|\F^{-1} \sigma_k\|_{L^1} \sum_{\substack{t \in \Z^n, \\ -3<t_i < 3, \\ i=1,\ldots, n}} \left( \sum_{k\in\Z^n} e^{|k|^{\frac{1}{s}}q} \left[ \sum_{l\in\Z^n} \|f_{t-(l-k)} g_l\|_{L^p} \right]^q \right)^{\frac{1}{q}}, 
		\end{eqnarray*}
		where we used Minkowski's and Young's inequality. By the basic properties of the decomposition function $\sigma_k$ and the shift invariance of this function it is trivial to see that
		\begin{equation*}
			\sup_{k\in \Z^n} \|\F^{-1} \sigma_k \|_{L^1} = \|\F^{-1} \sigma_0\|_{L^1} \leq  C_1.
		\end{equation*}
		Now we continue the norm estimate
		\begin{align*}
			\sup_{k\in\Z^n} \|\F^{-1} \sigma_k\|_{L^1} \sum_{\substack{t \in \Z^n, \\ -3<t_i < 3, \\ i=1,\ldots, n}} \left( \sum_{k\in\Z^n} e^{|k|^{\frac{1}{s}}q} \left[ \sum_{l\in\Z^n} \|f_{t-(l-k)} g_l\|_{L^p} \right]^q \right)^{\frac{1}{q}} & & \hspace{7cm}
		\end{align*}
		\begin{eqnarray*}
			& \leq & C_1 \sum_{\substack{t \in \Z^n, \\ -3<t_i < 3, \\ i=1,\ldots, n}} \left( \sum_{k\in\Z^n} e^{|k|^{\frac{1}{s}}q} \left[ \sum_{l\in\Z^n} \|f_{t-(l-k)} g_l\|_{L^p} \right]^q \right)^{\frac{1}{q}} \\
			& \leq & C_2 \max_{\substack{t \in \Z^n, \\ -3<t_i < 3, \\ i=1,\ldots, n}} \left( \sum_{k\in\Z^n} e^{|k|^{\frac{1}{s}}q} \left[ \sum_{l\in\Z^n} \|f_{t-(l-k)} g_l\|_{L^p} \right]^q \right)^{\frac{1}{q}} \\
			& \leq & C_2 \max_{\substack{t \in \Z^n, \\ -3<t_i < 3, \\ i=1,\ldots, n}} \left( \sum_{k\in\Z^n} e^{|k|^{\frac{1}{s}}q} \left[ \sum_{l\in\Z^n} \|f_{t-(l-k)}\|_{L^{2p}} \|g_l\|_{L^{2p}} \right]^q \right)^{\frac{1}{q}}
		\end{eqnarray*}
		by H\"older's inequality. The Lemmata \ref{nikolskij} and \ref{weightEstimate} yield 
		\begin{align*}
			C_2 \max_{\substack{t \in \Z^n, \\ -3<t_i < 3, \\ i=1,\ldots, n}} \left( \sum_{k\in\Z^n} e^{|k|^{\frac{1}{s}}q} \left[ \sum_{l\in\Z^n} \|f_{t-(l-k)}\|_{L^{2p}} \|g_l\|_{L^{2p}} \right]^q \right)^{\frac{1}{q}} \hspace{7cm}
		\end{align*}
		\begin{eqnarray*}
			& \leq & C_3 \max_{\substack{t \in \Z^n, \\ -3<t_i < 3, \\ i=1,\ldots, n}} \left( \sum_{k\in\Z^n} e^{|k|^{\frac{1}{s}}q} \left[ \sum_{l\in\Z^n} \|f_{t-(l-k)}\|_{L^{p}} \|g_l\|_{L^{p}} \right]^q \right)^{\frac{1}{q}} \\
			& \leq & C_3 \max_{\substack{t \in \Z^n, \\ -3<t_i < 3, \\ i=1,\ldots, n}} \Bigg( \sum_{k\in\Z^n} \Bigg[ \sum_{\substack{l\in\Z^n, \\ |l|\leq |l-k|}} e^{|l-k|^{\frac{1}{s}}} \|f_{t-(l-k)}\|_{L^{p}} e^{|l|^{\frac{1}{s}}} \|g_l\|_{L^{p}} e^{-\delta|l|^{\frac{1}{s}}} \\ 
			& & \qquad + \sum_{\substack{l\in\Z^n, \\ |l-k|\leq |l|}} e^{|l-k|^{\frac{1}{s}}} \|f_{t-(l-k)}\|_{L^{p}} e^{|l|^{\frac{1}{s}}} \|g_l\|_{L^{p}} e^{-\delta|l-k|^{\frac{1}{s}}} \Bigg]^q \Bigg)^{\frac{1}{q}} \\
			& \leq & C_3 \max_{\substack{t \in \Z^n, \\ -3<t_i < 3, \\ i=1,\ldots, n}} \Bigg( \sum_{k\in\Z^n} \Bigg[ \sum_{\substack{l\in\Z^n, \\ |l|\leq |l-k|}} e^{|l-k|^{\frac{1}{s}}} \|f_{t-(l-k)}\|_{L^{p}} e^{|l|^{\frac{1}{s}}} \|g_l\|_{L^{p}} e^{-\delta|l|^{\frac{1}{s}}} \Bigg]^q \Bigg)^{\frac{1}{q}} \\
			& & \qquad + C_3 \max_{\substack{t \in \Z^n, \\ -3<t_i < 3, \\ i=1,\ldots, n}} \Bigg( \sum_{k\in\Z^n} \Bigg[ \sum_{\substack{l\in\Z^n, \\ |l-k|\leq |l|}} e^{|l-k|^{\frac{1}{s}}} \|f_{t-(l-k)}\|_{L^{p}} e^{|l|^{\frac{1}{s}}} \|g_l\|_{L^{p}} e^{-\delta|l-k|^{\frac{1}{s}}} \Bigg]^q \Bigg)^{\frac{1}{q}} \\
			& = & S_1 + S_2.
		\end{eqnarray*}
		Now both parts $S_1$ and $S_2$ are estimated separately. Thus
		\begin{eqnarray*}
			S_1 & \stackrel{j=l-k}{=} & C_3 \max_{\substack{t \in \Z^n, \\ -3<t_i < 3, \\ i=1,\ldots, n}} \Bigg( \sum_{k\in\Z^n} \Bigg[ \sum_{\substack{j\in\Z^n, \\ |j+k|\leq |j|}} e^{|j|^{\frac{1}{s}}} \|f_{t-j}\|_{L^{p}} e^{|j+k|^{\frac{1}{s}}} \|g_{j+k}\|_{L^{p}} e^{-\delta|j+k|^{\frac{1}{s}}} \Bigg]^q \Bigg)^{\frac{1}{q}} \\
			& \stackrel{\frac{1}{q}+\frac{1}{q'}=1}{\leq} &  C_3 \max_{\substack{t \in \Z^n, \\ -3<t_i < 3, \\ i=1,\ldots, n}} \Bigg( \sum_{k\in\Z^n} \Bigg[ \Big( \sum_{\substack{j\in\Z^n, \\ |j+k|\leq |j|}} \Big| e^{|j|^{\frac{1}{s}}} \|f_{t-j}\|_{L^{p}} e^{|j+k|^{\frac{1}{s}}} \|g_{j+k}\|_{L^{p}} \Big|^q \Big)^{\frac{1}{q}} \\
			& & \qquad \qquad \Big( \sum_{\substack{j\in\Z^n, \\ |j+k|\leq |j|}} \Big| e^{-\delta|j+k|^{\frac{1}{s}}} \Big|^{q'} \Big)^{\frac{1}{q'}} \Bigg]^q \Bigg)^{\frac{1}{q}} \\
			& = & C_3 \max_{\substack{t \in \Z^n, \\ -3<t_i < 3, \\ i=1,\ldots, n}} \Bigg( \sum_{k\in\Z^n} \sum_{\substack{j\in\Z^n, \\ |j+k|\leq |j|}} e^{|j|^{\frac{1}{s}}q} \|f_{t-j}\|_{L^{p}}^q e^{|j+k|^{\frac{1}{s}}q} \|g_{j+k}\|_{L^{p}}^q \Big( \sum_{\substack{j\in\Z^n, \\ |j+k|\leq |j|}} e^{-\delta|j+k|^{\frac{1}{s}}q'} \Big)^{\frac{q}{q'}} \Bigg)^{\frac{1}{q}} \\
			& \stackrel{(\star)}{\leq} & C_4 \max_{\substack{t \in \Z^n, \\ -3<t_i < 3, \\ i=1,\ldots, n}} \Bigg( \sum_{k\in\Z^n} \sum_{\substack{j\in\Z^n, \\ |j+k|\leq |j|}} e^{|j|^{\frac{1}{s}}q} \|f_{t-j}\|_{L^{p}}^q e^{|j+k|^{\frac{1}{s}}q} \|g_{j+k}\|_{L^{p}}^q \Bigg)^{\frac{1}{q}} \\
			& \leq & C_4 \max_{\substack{t \in \Z^n, \\ -3<t_i < 3, \\ i=1,\ldots, n}} \Bigg( \sum_{j\in\Z^n} e^{|j|^{\frac{1}{s}}q} \|f_{t-j}\|_{L^{p}}^q \sum_{k\in\Z^n} e^{|j+k|^{\frac{1}{s}}q} \|g_{j+k}\|_{L^{p}}^q \Bigg)^{\frac{1}{q}} \\
			& \leq & C_4 \|g\|_{\mathcal{GM}_s^{p,q}} \max_{\substack{t \in \Z^n, \\ -3<t_i < 3, \\ i=1,\ldots, n}} \Bigg( \sum_{j\in\Z^n} e^{|j|^{\frac{1}{s}}q} e^{-|t-j|^{\frac{1}{s}}q} e^{|t-j|^{\frac{1}{s}}q} \|f_{t-j}\|_{L^{p}}^q \Bigg)^{\frac{1}{q}}. 
		\end{eqnarray*}
		The estimate $(\star)$ needs a more detailed consideration, i.e., we have to show that for every fixed $k\in\Z^n$ there exists a positive constant $\tilde{C}$ such that 
		\begin{equation} \label{sumEst}
			\left( \sum_{\substack{j\in\Z^n, \\ |j+k|\leq |j|}} e^{-\delta|j+k|^{\frac{1}{s}}q'} \right)^{\frac{q}{q'}} \leq \tilde{C}.
		\end{equation}
		We prove this by induction over the dimension $n$. For $n=1$ it obviously holds
		\[ \left( \sum_{\substack{j\in\Z, \\ |j+k|\leq |j|}} e^{-\delta|j+k|^{\frac{1}{s}}q'} \right)^{\frac{q}{q'}} \leq \left( \sum_{j=-\infty}^\infty e^{-\delta|j+k|^{\frac{1}{s}}q'} \right)^{\frac{q}{q'}} \leq \tilde{C} \]
		because $k\in\Z$ is just a shift in the argument of the exponential function. \\
		Now we assume that (\ref{sumEst}) holds for the $l$-dimensional case. Then we need to show
		\[ \left( \sum_{\substack{j\in\Z^{l+1}, \\ |j+k|\leq |j|}} e^{-\delta|j+k|^{\frac{1}{s}}q'} \right)^{\frac{q}{q'}} \leq \tilde{C}. \]
		This can be done by the following computations
		\begin{align*}
			\sum_{\substack{j\in\Z^{l+1}, \\ |j+k|\leq |j|}} e^{-\delta|j+k|^{\frac{1}{s}}q'} \leq \sum_{j\in\Z^{l+1}} e^{-\delta|j+k|^{\frac{1}{s}}q'} & & \hspace{7cm}
		\end{align*}
		\begin{eqnarray*}
			& = & \sum_{j_1=-\infty}^\infty \sum_{j_2=-\infty}^\infty \ldots \sum_{j_l=-\infty}^\infty \sum_{j_{l+1}=-\infty}^\infty e^{-\delta ( (j_1+k_1)^2 + (j_2+k_2)^2 + \ldots + (j_l+k_l)^2 + (j_{l+1}+k_{l+1})^2 )^{\frac{1}{2s}}q'} \\
			& = & \sum_{j_1=-\infty}^\infty \sum_{j_2=-\infty}^\infty \ldots \sum_{j_l=-\infty}^\infty \sum_{j_{l+1}=-\infty}^\infty e^{-\frac{\delta}{2} ( (j_1+k_1)^2 + (j_2+k_2)^2 + \ldots + (j_l+k_l)^2 + (j_{l+1}+k_{l+1})^2 )^{\frac{1}{2s}}q'} \\
			& & \qquad \quad e^{-\frac{\delta}{2} ( (j_1+k_1)^2 + (j_2+k_2)^2 + \ldots + (j_l+k_l)^2 + (j_{l+1}+k_{l+1})^2 )^{\frac{1}{2s}}q'} \\
			& \leq & \sum_{j_1=-\infty}^\infty \sum_{j_2=-\infty}^\infty \ldots \sum_{j_l=-\infty}^\infty \sum_{j_{l+1}=-\infty}^\infty e^{-\frac{\delta}{2} ( (j_1+k_1)^2 + (j_2+k_2)^2 + \ldots + (j_l+k_l)^2 )^{\frac{1}{2s}}q'} e^{-\frac{\delta}{2} ( (j_{l+1}+k_{l+1})^2 )^{\frac{1}{2s}}q'} \\
			& = & \sum_{j\in\Z^l} e^{-\frac{\delta}{2} |j+k|^{\frac{1}{s}}q'} \, \sum_{j_{l+1}=-\infty}^\infty e^{-\frac{\delta}{2} ( (j_{l+1}+k_{l+1})^2 )^{\frac{1}{2s}}q'} \\
			& \leq & \tilde{C}.
		\end{eqnarray*}
		Note that in each estimate the choice of constant $\tilde{C}$ can be different. Thus we verified the estimate $(\star)$. \\
		
		By the consideration
		\[ \max_{\substack{t \in \Z^n, \\ -3<t_i < 3, \\ i=1,\ldots, n}} \sup_{j\in\Z^n} e^{q(|j|^{\frac{1}{s}}-|t-j|^{\frac{1}{s}})} \leq \max_{\substack{t \in \Z^n, \\ -3<t_i < 3, \\ i=1,\ldots, n}} e^{q|t|^{\frac{1}{s}}} \leq C, \]
		where we used the triangle inequality and a positive constant $C$, we finally obtain for the upper estimate
		\[ C_4 \|g\|_{\mathcal{GM}_s^{p,q}} \max_{\substack{t \in \Z^n, \\ -3<t_i < 3, \\ i=1,\ldots, n}} \left( \sum_{j\in\Z^n} e^{|j|^{\frac{1}{s}}q} e^{-|t-j|^{\frac{1}{s}}q} e^{|t-j|^{\frac{1}{s}}q} \|f_{t-j}\|_{L^{p}}^q \right)^{\frac{1}{q}} \leq C_5 \|g\|_{\mathcal{GM}_s^{p,q}} \|f\|_{\mathcal{GM}_s^{p,q}}. \]
		For the second sum $S_2$ the estimate
		\[ S_2 \leq \tilde{C}_5 \|g\|_{\mathcal{GM}_s^{p,q}} \|f\|_{\mathcal{GM}_s^{p,q}} \]
		follows by analogous computations. \\
		Remark that all computations can also be done by taking the $l^\infty$ and $L^\infty$ norm, respectively. Therefore the algebra property follows for $p=q=\infty$ and the proof is completed.
	\end{proof}
	
	\subsection{Non-analytic Superposition}
	The idea of introducing the subsequent composition operators is given in \cite{brs}. \\
	The first tool will be the subalgebra property for Gevrey-modulation spaces $\mathcal{GM}^{p,q}_{s}$. Therefore we find the following decomposition of the phase space. Let $R>0$ and $\epsilon=(\epsilon_1,\ldots, \epsilon_n)$ be fixed with $\epsilon_j \in \{0,1\}$, $j=1,\ldots,n$. Then we have a decomposition of $\R^n$ into $(2^n+1)$ parts
\[ P_R = \{ \xi\in\R^n : |\xi_j|\leq R, j=1,\ldots, n \} \]
and
\[ P_R(\epsilon) = \{ \xi\in \R^n: \sgn (\xi_j) = (-1)^{\epsilon_j}, j=1,\ldots,n \} \setminus P_R. \]
\begin{prop} \label{subalgebra}
	Let $1\leq p, q \leq \infty$. Suppose that $\epsilon=(\epsilon_1,\ldots, \epsilon_n)$ is fixed with $\epsilon_j \in \{0,1\}$, $j=1,\ldots,n$. Given $s>1$ and $R>0$ the spaces
	\[ \mathcal{GM}^{p,q}_{s}(\epsilon, R) = \{ f\in \mathcal{GM}^{p,q}_{s}(\R^n): \supp \F(f) \subset P_R(\epsilon) \} \]
	are subalgebras of $\mathcal{GM}^{p,q}_{s}$. Furthermore it holds
	\[ \|f g\|_{\mathcal{GM}^{p,q}_{s}(\R^n)} \leq D \|f\|_{\mathcal{GM}^{p,q}_{s}(\R^n)} \|g\|_{\mathcal{GM}^{p,q}_{s}(\R^n)} \]
	for all $f,g \in\mathcal{GM}^{p,q}_{s}(\epsilon, R)$. The constant $D$ can be specified by
	\[ D = C_0 \left( s \omega_n (\delta q')^{-sn} \int_{\delta q' R^{\frac{1}{s}}}^\infty y^{sn-1} e^{-y} \, dy \right)^{\frac{1}{q'}}, \]
	where the constant $C_0>0$ is only dependent on $p$ and $n$ and $q'$ is chosen such that $\frac{1}{q}+\frac{1}{q'}=1$.
\end{prop}
\begin{proof}
	Let $f,g \in\mathcal{GM}^{p,q}_{s}(\epsilon, R)$. By 
	\[ \supp (\F f \ast \F g) \subset \{ \xi +\eta: \xi \in \supp \F (f), \eta \in \supp \F (g) \} \]
	we have $\supp \F (fg) \subset P_R(\epsilon)$. In order to show the algebra property we follow the same steps as in the proof of Theorem \ref{algebra}. Consider there the term
	\begin{align*}
		C_3 \max_{\substack{t \in \Z^n, \\ -3<t_i < 3, \\ i=1,\ldots, n}} \Bigg( \sum_{k\in\Z^n\cap P_R(\epsilon)} \sum_{\substack{j\in\Z^n\cap P_R(\epsilon), \\ |j+k|\leq |j|}} e^{|j|^{\frac{1}{s}}q} \|f_{t-j}\|_{L^{p}}^q e^{|j+k|^{\frac{1}{s}}q} \|g_{j+k}\|_{L^{p}}^q \Big( \sum_{\substack{j\in\Z^n\cap P_R(\epsilon), \\ |j+k|\leq |j|}} e^{-\delta|j+k|^{\frac{1}{s}}q'} \Big)^{\frac{q}{q'}} \Bigg)^{\frac{1}{q}}, 
	\end{align*}
	where we used the support properties of the functions $f$ and $g$. Now we estimate the sum
	\[ S_1 = \sum_{\substack{j\in\Z^n\cap P_R(\epsilon), \\ |j+k|\leq |j|}} e^{-\delta|j+k|^{\frac{1}{s}}q'}. \]
	Recall that $\epsilon$ is fixed. The essential fact we are using is that for $j,k \in P_R(\epsilon)$ both of the indices have the same sign, i.e., $\sgn(j_i) = \sgn(k_i)$ for $i=1,\ldots, n$. Thus, the case $|j+k|\leq |j|$ is only fulfilled for $k=0$. But this is the trivial case and also included in the considerations of $S_2$. So the other sum which needs to be estimated is 
	\[ S_2 = \sum_{\substack{j\in\Z^n\cap P_R(\epsilon), \\ |j|\leq |j+k|}} e^{-\delta|j|^{\frac{1}{s}}q'}. \]
	The same arguments yield the correctness of $|j|\leq |j+k|$ for all indices $j, k \in \Z^n$. Therefore it follows for $k\in P_R(\epsilon)$
	\begin{eqnarray*}
		S_2 & = & \sum_{j\in\Z^n\cap P_R(\epsilon)} e^{-\delta|j|^{\frac{1}{s}}q'} \\
			& \leq & \int_{|\tau|\geq R} e^{-\delta|\tau|^{\frac{1}{s}}q'} \, d\tau \\
			& = & \omega_n \int_R^\infty r^{n-1} e^{-\delta q' r^{\frac{1}{s}}} \, dr \\
			& \stackrel{[t=r^{\frac{1}{s}}]}{=} & s \omega_n \int_{R^{\frac{1}{s}}}^\infty t^{sn-1} e^{-\delta q' t} \, dt \\
			& \stackrel{[y=\delta q' t]}{=} & s \omega_n (\delta q')^{-sn} \int_{\delta q' R^{\frac{1}{s}}}^\infty y^{sn-1} e^{-y} \, dy.
	\end{eqnarray*}
	The proof is completed.
\end{proof}

Note that in the following we assume every function to be real-valued unless it is explicitly stated that complex functions are allowed. 
\begin{lma} \label{expfunc}
	Let $1 \leq p\leq \infty$. Then
	\[ \|e^{\i u} -1\|_{L^p(\R^n)} \leq C \|u\|_{L^p(\R^n)} \]
	with a constant $C>0$. 
\end{lma}
\begin{proof}
	Consider the function $f(t) = e^{\i t}-1$. By applying the mean value theorem we obtain
	\[ f(u) - f(0) = e^{\i u}-1 = u f'(\xi), \quad \xi\in (0,u). \]
	Note that $f'(t) = -\sin(t) + \i \cos(t)$ is a bounded function. Therefore it follows
	\begin{eqnarray*}
		\|f\|^p_{L^p(\R^n)} & = & \int_{\R^n} |e^{\i t} -1|^p \, dt \\
			& \leq & C \int_{\R^n} |u|^p \, dt \\
			& = & C \|u\|^p_{L^p(\R^n)},
	\end{eqnarray*}
	which completes the proof.
\end{proof}
	
	In order to establish the next result we need to recall two lemmata from \cite{brs}. 
	\begin{lma} \label{lma46brs}
		Let $N\in \N$ and suppose $a_1, a_2,\ldots, a_N$ to be complex numbers. Then it holds
		\[ a_1 \cdot a_2 \cdot \ldots \cdot a_N -1  = \sum_{l=1}^{N} \sum_{\substack{ j=(j_1,\ldots, j_l), \\ 0\leq j_1 < \ldots < j_l \leq N }} (a_{j_1} -1) \cdot \ldots \cdot (a_{j_l}-1). \]
	\end{lma}
	\begin{proof}
		Cf. Lemma 4.6. in \cite{brs}.
	\end{proof}
	
	\begin{lma} \label{lma45brs}
		Let $\alpha>0$. Assume $f=f(t)$ to be the function
		\[ f(t) = \int_{t}^\infty e^{-y} y^{\alpha-1} \, dy, \qquad t\geq 0. \]
		The inverse $g$ of the function $f$ maps $(0,\Gamma(\alpha)]$ onto $[0,\infty)$ and it holds
		\[ \lim_{u\to 0} \frac{g(u)}{\log \frac{1}{u}} = 1. \]
	\end{lma}
	\begin{proof}
		Cf. Lemma 4.5. in \cite{brs}.
	\end{proof}
		
	\begin{lma} \label{estSuper}
		Let $s>1$ and $1\leq p,q <\infty$. Suppose $u \in \mathcal{GM}_s^{p,q}$. Then it holds
		\[ \| e^{iu}-1\|_{\mathcal{GM}_s^{p,q}} \leq c \begin{cases} e^{b \|u\|^{\frac{1}{s}}_{\mathcal{GM}_s^{p,q}}\log \|u\|_{\mathcal{GM}_s^{p,q}}} , \quad \mbox{ if } \|u\|_{\mathcal{GM}_s^{p,q}} > 1 \\
																		\|u\|_{\mathcal{GM}_s^{p,q}}, \qquad \qquad \qquad \mbox{ if } \|u\|_{\mathcal{GM}_s^{p,q}} \leq 1
														\end{cases} \]
		with constants $b,c>0$ independent of $u$.
	\end{lma}
	\begin{proof}
		This proof basically follows the same steps as in the proof of Theorem 2.3 in \cite{brs}. \\
		Let $u\in \mathcal{GM}_s^{p,q}$ satisfying $\supp \F (u) \subset P_R$. \\
		First of all we find the following representation by Taylor expansion
		\[ e^{\i u}-1 = \sum_{l=1}^r \frac{(\i u)^l}{l!} + \sum_{l=r+1}^{\infty} \frac{(\i u)^l}{l!} \]
		with the norm estimate
		\[ \| e^{\i u}-1\|_{\mathcal{GM}_s^{p,q}} \leq \Big\| \sum_{l=1}^r \frac{(\i u)^l}{l!} \Big\|_{\mathcal{GM}_s^{p,q}} + \Big\| \sum_{l=r+1}^{\infty} \frac{(\i u)^l}{l!} \Big\|_{\mathcal{GM}_s^{p,q}} = S_1 + S_2. \]
		By Theorem \ref{algebra} we obtain
		\[ S_2 \leq \sum_{l=r+1}^\infty \frac{1}{l!} \|u^l\|_{\mathcal{GM}_s^{p,q}} \leq \frac{1}{C} \sum_{l=r+1}^{\infty} \frac{ (C\|u\|_{\mathcal{GM}_s^{p,q}})^l}{l!}. \]
		Now we choose $r$ as a function of $\|u\|_{\mathcal{GM}_s^{p,q}}$.
		\begin{enumerate}
			\item $C\|u\|_{\mathcal{GM}_s^{p,q}}>1$. Assume that 
			\[ 3C\|u\|_{\mathcal{GM}_s^{p,q}} \leq r \leq 3C \|u\|_{\mathcal{GM}_s^{p,q}} +1 \]
			and recall Stirling's formula $l! = \Gamma(l+1) \geq l^l e^{-l} \sqrt{2\pi l}$. Thus, we get
			\begin{eqnarray*}
				\sum_{l=r+1}^{\infty} \frac{(C\|u\|_{\mathcal{GM}_s^{p,q}})^l}{l!} & \leq & \sum_{l=r+1}^{\infty} \left( \frac{r}{l} \right)^l \left(\frac{e}{3}\right)^l \frac{1}{\sqrt{2\pi l}} \\
					& \leq & \sum_{l=r+1}^\infty \left(\frac{e}{3} \right)^l \\
					& \leq & \frac{3}{3-e}.
			\end{eqnarray*}
			\item $C\|u\|_{\mathcal{GM}_s^{p,q}}\leq 1$. It follows
			\begin{eqnarray*}
				 \sum_{l=r+1}^{\infty} \frac{(C\|u\|_{\mathcal{GM}_s^{p,q}})^l}{l!} & \leq & \sum_{l=1}^{\infty} \frac{(C\|u\|_{\mathcal{GM}_s^{p,q}})^l}{l!} \\
					& \leq & C e \|u\|_{\mathcal{GM}_s^{p,q}}.
			\end{eqnarray*}
		\end{enumerate}
		Now an appropriate estimate for $S_1$ needs to be shown. \\
		We have
		\begin{align*}
		S_1 = \Big\| \sum_{l=1}^r \frac{(\i u)^l}{l!} \Big\|_{\mathcal{GM}^{p,q}_{s}(\R^n)} = \Bigg( \sum_{k\in\Z^n} e^{|k|^{\frac{1}{s}}q} \Big\|\Box_k \Big(\sum_{l=1}^{r} \frac{(\i u)^l}{l!} \Big) \Big\|_{L^p}^q \Bigg)^{\frac{1}{q}} 
	\end{align*}
	\begin{eqnarray*}
		& = & \Bigg( \sum_{\substack{k\in\Z^n, \\ -Rr-1<k_i< Rr+1, \\ i=1,\ldots,n}} e^{|k|^{\frac{1}{s}}q} \Big\|\Box_k \Big(\sum_{l=1}^{r} \frac{(\i u)^l}{l!} \Big) \Big\|_{L^p}^q \Bigg)^{\frac{1}{q}} \\
			& \leq & \Bigg( \sum_{\substack{k\in\Z^n, \\ -Rr-1<k_i< Rr+1, \\ i=1,\ldots,n}} e^{|k|^{\frac{1}{s}}q} \|\Box_k (e^{\i u}-1) \|_{L^p}^q \Bigg)^{\frac{1}{q}} + S_2,
	\end{eqnarray*}
where $S_2$ converges as shown in the previous considerations. Estimating the first sum due to Lemma \ref{expfunc}, Lemma \ref{bernstein} and Minkowski's inequality we get
	\begin{align*}
		\Bigg( \sum_{\substack{k\in\Z^n, \\ -Rr-1<k_i< Rr+1, \\ i=1,\ldots,n}} e^{|k|^{\frac{1}{s}}q} \| \Box_k (e^{\i u}-1) \|_{L^p}^q \Bigg)^{\frac{1}{q}} & & \hspace{7cm}
	\end{align*}
	\begin{eqnarray*}
		& = &	\Bigg( \sum_{\substack{k\in\Z^n, \\ -Rr-1<k_i< Rr+1, \\ i=1,\ldots,n}} e^{|k|^{\frac{1}{s}}q} \Big\|\Box_k \big(e^{\i (\sum_{l\in \Z^n} \Box_l u)}-1 \big) \Big\|_{L^p}^q \Bigg)^{\frac{1}{q}} \\
		& \leq & C_0 \Bigg( \sum_{\substack{k\in\Z^n, \\ -Rr-1<k_i< Rr+1, \\ i=1,\ldots,n}} e^{|k|^{\frac{1}{s}}q} \Big\|e^{\i (\sum_{l\in \Z^n} \Box_l u)}-1 \Big\|_{L^p}^q \Bigg)^{\frac{1}{q}} \\
		& \leq & C_0 \Bigg( \sum_{\substack{k\in\Z^n, \\ -Rr-1<k_i< Rr+1, \\ i=1,\ldots,n}} e^{|k|^{\frac{1}{s}}q} \Big\|\sum_{l\in \Z^n} \Box_l u \Big\|_{L^p}^q \Bigg)^{\frac{1}{q}} \\
		& \leq & C_0 \Bigg( \sum_{\substack{k\in\Z^n, \\ -Rr-1<k_i< Rr+1, \\ i=1,\ldots,n}} e^{|k|^{\frac{1}{s}}q} \Big( \sum_{\substack{l\in\Z^n, \\ |k|\leq |l|}} \|\Box_l u \|_{L^p} \Big)^q \Bigg)^{\frac{1}{q}} \\
		& & \quad + C_0 \Bigg( \sum_{\substack{k\in\Z^n, \\ -Rr-1<k_i< Rr+1, \\ i=1,\ldots,n}} e^{|k|^{\frac{1}{s}}q} \Big( \sum_{\substack{l\in\Z^n, \\ |l|\leq |k|}} \|\Box_l u \|_{L^p} \Big)^q \Bigg)^{\frac{1}{q}} \\
		& = & T_1 + T_2.
	\end{eqnarray*}
Now we consider the terms $T_1$ and $T_2$ separately. Thus, we obtain
	\begin{eqnarray*}
		T_1 & \leq & C_0 e^{(\sqrt{n} (Rr+1))^{\frac{1}{s}}} \Bigg( \sum_{\substack{k\in\Z^n, \\ -Rr-1<k_i< Rr+1, \\ i=1,\ldots,n}} e^{-(\sqrt{n} (Rr+1))^{\frac{1}{s}}q} e^{|k|^{\frac{1}{s}}q} \Big( \sum_{\substack{l\in\Z^n, \\ |k|\leq |l|}} \|\Box_l u \|_{L^p} \Big)^q \Bigg)^{\frac{1}{q}} \\
			& \leq & C_0 e^{(\sqrt{n} (Rr+1))^{\frac{1}{s}}} \Bigg( \sum_{\substack{k\in\Z^n, \\ -Rr-1<k_i< Rr+1, \\ i=1,\ldots,n}} e^{-|k|^{\frac{1}{s}}q} e^{\frac{|k|^{\frac{1}{s}}}{2}q} \Big( \sum_{\substack{l\in\Z^n, \\ |k|\leq |l|}} e^{\frac{|k|^{\frac{1}{s}}}{2}} e^{-\frac{|l|^{\frac{1}{s}}}{2}} e^{-\frac{|l|^{\frac{1}{s}}}{2}} e^{|l|^{\frac{1}{s}}} \|\Box_l u \|_{L^p} \Big)^q \Bigg)^{\frac{1}{q}} 
	\end{eqnarray*}
	\begin{eqnarray*}
			& \leq & C_0 e^{(\sqrt{n} (Rr+1))^{\frac{1}{s}}} \left( \sum_{k\in\Z^n} e^{-\frac{|k|^{\frac{1}{s}}}{2}q} \right)^{\frac{1}{q}} \sum_{l\in\Z^n} e^{-\frac{|l|^{\frac{1}{s}}}{2}} e^{|l|^{\frac{1}{s}}} \|\Box_l u \|_{L^p} \\
			& \leq & C_1 e^{(\sqrt{n} (Rr+1))^{\frac{1}{s}}} \left( \sum_{l\in\Z^n} e^{-\frac{|l|^{\frac{1}{s}}}{2}q'} \right)^{\frac{1}{q'}} \left( \sum_{l\in\Z^n} e^{|l|^{\frac{1}{s}}q} \|\Box_l u \|_{L^p}^q \right)^{\frac{1}{q}} \\
			& \leq & C_2 e^{(\sqrt{n} (Rr+1))^{\frac{1}{s}}} \|u\|_{\mathcal{GM}^{p,q}_{s}(\R^n)},
	\end{eqnarray*}
where $C_2>0$ is a constant depending on the dimension $n$, the weight parameter $s$ and the integrability parameter $q$. Furthermore we get 
	\begin{eqnarray*}
		T_2 & \leq & C_0 e^{2(\sqrt{n} (Rr+1))^{\frac{1}{s}}} \Bigg( \sum_{\substack{k\in\Z^n, \\ -Rr-1<k_i< Rr+1, \\ i=1,\ldots,n}} e^{2(\sqrt{n} (Rr+1))^{\frac{1}{s}}q} e^{|k|^{\frac{1}{s}}q} \Big( \sum_{\substack{l\in\Z^n, \\ |l|\leq |k|}} \|\Box_l u \|_{L^p} \Big)^q \Bigg)^{\frac{1}{q}} \\
			& \leq & C_0 e^{2(\sqrt{n} (Rr+1))^{\frac{1}{s}}} \Bigg( \sum_{\substack{k\in\Z^n, \\ -Rr-1<k_i< Rr+1, \\ i=1,\ldots,n}} e^{-2|k|^{\frac{1}{s}}q} e^{|k|^{\frac{1}{s}}q} \Big( \sum_{\substack{l\in\Z^n, \\ |l|\leq |k|}} \|\Box_l u \|_{L^p} \Big)^q \Bigg)^{\frac{1}{q}} \\
			& \leq & \tilde{C}_1 e^{2(\sqrt{n} (Rr+1))^{\frac{1}{s}}} \sum_{l\in\Z^n} e^{-|l|^{\frac{1}{s}}} e^{|l|^{\frac{1}{s}}} \|\Box_l u \|_{L^p} \\
			& \leq & \tilde{C}_1 e^{2(\sqrt{n} (Rr+1))^{\frac{1}{s}}} \left( \sum_{l\in\Z^n} e^{-|l|^{\frac{1}{s}}q'} \right)^{\frac{1}{q'}} \left( \sum_{l\in\Z^n} e^{|l|^{\frac{1}{s}}q} \|\Box_l u \|_{L^p}^q \right)^{\frac{1}{q}} \\
			& \leq & \tilde{C}_2 e^{2(\sqrt{n} (Rr+1))^{\frac{1}{s}}} \|u\|_{\mathcal{GM}^{p,q}_{s}(\R^n)},
	\end{eqnarray*}
where $\tilde{C}_2>0$ is a constant depending on the same parameters as $C_2$ in the previous estimate. Summarizing we have shown that
\[ \Big\|\sum_{l=1}^r \frac{(\i u)^l}{l!} \Big\|_{\mathcal{GM}^{p,q}_{s}} \leq c_0 e^{2(\sqrt{n}(Rr+1))^{\frac{1}{s}} } \|u\|_{\mathcal{GM}^{p,q}_{s}} + c_1 \|u\|_{\mathcal{GM}^{p,q}_{s}} \]
with positive constants $c_0, c_1$ depending on $n$, $s$ and $q$. \\
	Up to now we have proved
	\begin{equation} \label{expest}
		\| e^{\i u}-1 \|_{\mathcal{GM}^{p,q}_{s}} \leq c \|u\|_{\mathcal{GM}^{p,q}_{s}} \left( 1 + e^{b R^{\frac{1}{s}} \| u\|_{\mathcal{GM}^{p,q}_{s}}^{\frac{1}{s}} } \right)
	\end{equation}
	for admissible $u\in \mathcal{GM}^{p,q}_{s}(\R^n)$ with positive constants $b,c$ depending on $n,s$ and $q$ but independent of $u, r$ and $R$. \\
	The next step consists of choosing a general $u\in \mathcal{GM}^{p,q}_{s}(\R^n)$. For decomposing $u$ on the phase space we introduce two functions $\chi_{R,\epsilon}$ and $\chi_R$, that is the characteristic function of $P_R(\epsilon)$ and the characteristic function of the set $P_R$, respectively. By defining 
	\begin{eqnarray*}
		u_\epsilon (x) & = & \F^{-1} [\chi_{R,\epsilon} (\xi) (\F u)(\xi) ] (x) , \\
		u_0 (x) & = & \F^{-1} [\chi_R (\xi) (\F u) (\xi) ] (x)
	\end{eqnarray*}
	we can rewrite $u$ as 
	\begin{equation} \label{urepr}
		u(x) = u_0(x) + \sum_{\epsilon\in I} u_\epsilon (x) , 
	\end{equation}
	where $I$ is the set of all $\epsilon=(\epsilon_1,\ldots, \epsilon_n)$ with $\epsilon_j \in \{0,1\}$, $j=1,\ldots,n$. This representation obviously gives
	\begin{equation} \label{normrepr}
		\| u\|_{\mathcal{GM}^{p,q}_{s}(\R^n)} = \| u_0\|_{\mathcal{GM}^{p,q}_{s}(\R^n)}  + \sum_{\epsilon\in I} \| u_\epsilon \|_{\mathcal{GM}^{p,q}_{s}(\R^n)} 
	\end{equation}
	taking into account the support property. Due to representation (\ref{urepr}) and Lemma \ref{lma46brs} we obtain
	\[ e^{\i u} -1 = \sum_{l=1}^{2^n +1} \sum_{0\leq j_1 <\ldots < j_l \leq 2^n} (e^{\i u_{j_1}} -1)\cdot \ldots \cdot (e^{\i u_{j_l}} -1) \]
	by using an appropriate numbering. Theorem \ref{algebra} immediately yields
	\[ \|e^{\i u} -1\|_{\mathcal{GM}^{p,q}_{s}} \leq C^{l-1} \sum_{l=1}^{2^n +1} \sum_{0\leq j_1 <\ldots < j_l \leq 2^n} \|e^{\i u_{j_1}} -1 \|_{\mathcal{GM}^{p,q}_{s}} \cdot \ldots \cdot \| e^{\i u_{j_l}} -1 \|_{\mathcal{GM}^{p,q}_{s}}. \] 
	By Proposition \ref{subalgebra}, \eqref{normrepr} and \eqref{expest} it follows
	\begin{eqnarray}
		\| e^{\i u_{j_k}} -1\|_{\mathcal{GM}^{p,q}_{s}} & \leq & \frac{1}{D} \left(e^{D \|u_{j_k}\|_{\mathcal{GM}^{p,q}_{s}}} -1 \right) \nonumber \\
			& \leq & \frac{1}{D} \left( e^{D \|u\|_{\mathcal{GM}^{p,q}_{s}}} -1 \right), \label{est1} \\
		\| e^{\i u_0} -1\|_{\mathcal{GM}^{p,q}_{s}} & \leq &  c \|u\|_{\mathcal{GM}^{p,q}_{s}} \left( 1 + e^{b R^{\frac{1}{s}} \| u\|_{\mathcal{GM}^{p,q}_{s}}^{\frac{1}{s}} } \right) \label{est2} 
	\end{eqnarray}
	for an admissible choice of $j_k$. Note that for $\|u\|_{\mathcal{GM}^{p,q}_{s}(\R^n)} \leq 1$ our considerations are getting trivial. Therefore we restrict ourselves to the case $\|u\|_{\mathcal{GM}^{p,q}_{s}(\R^n)} >1$. \\
	The final step of the proof is to choose the number $R$ as a function of $\|u\|_{\mathcal{GM}^{p,q}_{s}}$ such that \eqref{est1} and \eqref{est2} will be  approximately of the same size. As mentioned in Proposition \ref{subalgebra} we know that the algebra constant $D$ in \eqref{est1} is a function of $R$, i.e.,
	\[ D = D(R) = C_0 \left( s \omega_n (\delta q')^{-sn} \int_{\delta q' R^{\frac{1}{s}}}^\infty y^{sn-1} e^{-y} \, dy \right)^{\frac{1}{q'}}. \]
	Taking into account that 
	\begin{itemize}
		\item $D$ is strictly monotone positive,
		\item $D(0)>1$ and
		\item $\displaystyle \lim_{R\to \infty} D(R) = 0$
	\end{itemize}
	we can easily set $D(R) = \|u\|_{\mathcal{GM}^{p,q}_{s}(\R^n)}^{\frac{1}{s}-1}$. In view of Lemma \ref{lma45brs} this gives
	\[ \left( c f(\delta q' R^{\frac{1}{s}}) \right)^{\frac{1}{q'}} = \|u\|_{\mathcal{GM}^{p,q}_{s}(\R^n)}^{\frac{1}{s}-1}. \]
	Thus, by Lemma \ref{lma45brs} it follows
	\[ R = (\delta q')^{-s} \left( g \left(\frac{\|u\|_{\mathcal{GM}^{p,q}_{s}}^{q'(\frac{1}{s}-1)}}{c} \right) \right)^s \]
	and moreover,
	\[ R \leq C \left(\log c +(q'-\frac{q'}{s}) \log \|u\|_{\mathcal{GM}^{p,q}_{s}} \right)^s. \]
	Note that the constants $C$ and $c$ are independent of $u$. This together with (\ref{est1}) and (\ref{est2}) gives the desired result and completes the proof. 
\end{proof}
	
	\begin{lma} \label{ContExp}
		Assume $u\in \mathcal{GM}_s^{p,q}$ to be fixed and define a function $g: \R \mapsto \mathcal{GM}_s^{p,q}$ by $g(\xi) = e^{\i u(x) \xi}-1$. Then the function $g$ is continuous.
	\end{lma}
	\begin{proof}
		By the identity 
		\[ e^{\i \xi} - e^{\i \eta} = (e^{\i \eta} - 1)(e^{\i (\xi-\eta)} -1) + (e^{\i (\xi-\eta)} -1), \]
		applying the algebra property and Lemma \ref{estSuper} it follows
		\begin{eqnarray*}
			\|e^{\i u(\cdot) \xi} - e^{\i u(\cdot) \eta} \|_{\mathcal{GM}_s^{p,q}} & = & \|(e^{\i u(\cdot) \eta} - 1)(e^{\i u(\cdot) (\xi-\eta)} -1) + (e^{\i u(\cdot) (\xi-\eta)} -1) \|_{\mathcal{GM}_s^{p,q}} \\
				& \leq & \big(C \|e^{\i u(\cdot) \eta} - 1\|_{\mathcal{GM}_s^{p,q}} + 1 \big) \|e^{\i u(\cdot) (\xi-\eta) } - 1\|_{\mathcal{GM}_s^{p,q}} \\
				& \leq & \big(C_0 e^{b \|u(\cdot) \eta\|^{\frac{1}{s}}_{\mathcal{GM}_s^{p,q}}\log \|u(\cdot) \eta\|_{\mathcal{GM}_s^{p,q}}} + 1 \big) \| u(\cdot) (\xi -\eta) \|_{\mathcal{GM}_s^{p,q}} \\
				& \leq & C_1 \|u\|_{\mathcal{GM}_s^{p,q}} |\xi-\eta | \\
				& \leq & C_2 |\xi -\eta|. 
		\end{eqnarray*}
		Since we want to show continuity we have $|\xi-\eta|<\delta$ with sufficiently small $\delta$. Therefore it is natural to assume $\|u(\cdot) (\xi- \eta)\|_{\mathcal{GM}_s^{p,q}} \leq 1$. The proof is completed.
	\end{proof}
	
	Now we can establish the following result, where basically ideas of \cite{brs} are used. 
	\begin{thm} \label{Superposition}
		Let the weight parameter $s>1$ and $\mu$ be a complex measure on $\R$ such that 
		\begin{equation} \label{FourierEst}
			L_1(\lambda) = \int_{\R} e^{\lambda (|\xi|^{\frac{1}{s}} \log |\xi| )} \, d|\mu| (\xi) < \infty 
		\end{equation}
		for any $\lambda >0 $ and such that $\mu(\R) = 0$. \\
		Furthermore assume that the function $f$ is the inverse Fourier transform of $\mu$. Then $f\in C^\infty$ and the composition operator $T_f: u \mapsto f \circ u$ maps $\mathcal{GM}_s^{p,q}$ into $\mathcal{GM}_s^{p,q}$.
	\end{thm}
	\begin{proof}
		Equation \eqref{FourierEst} yields $\int_{\R} d|\mu|(\xi) < \infty$. Thus $\mu$ is a finite measure and $\mu(\R) =0$ makes sense. Now we define the inverse Fourier transform of $\mu$ 
		\[ f(t) = \frac{1}{\sqrt{2\pi}} \int_{\R} e^{\i \xi t} \, d \mu(\xi). \]
		Moreover $\int_{\R} |(\i \xi)^j| \, d|\mu|(\xi) <\infty$ is deduced from equation \eqref{FourierEst} for all $j\in\N$. This gives $f\in C^\infty$ and due to $\mu(\R)=0$ we can also write $f$ as follows
		\[ f(t) = \frac{1}{\sqrt{2\pi}} \int_{\R} (e^{\i \xi t}-1) \, d \mu(\xi). \]
		Since $\mu$ is a complex measure we can split it up into real part $\mu_r$ and imaginary part $\mu_i$ where each of them is a signed measure. Thus we have $\mu(E) = \mu_r(E) + \i \mu_i(E)$ for all measurable sets $E$. By Jordan decomposition we obtain
		\[ \mu(E) = \mu_r^+(E) - \mu_r^-(E) +\i (\mu_i^+(E) - \mu_i^-(E)). \]
		Here it is $\mu^+(E) = \mu(E\cap P)$ and $\mu^-(E) = -\mu(E\cap N)$, where the set $P$ is a positive set for $\mu$ and $N$ is a negative set for $\mu$. Without loss of generality we proceed our computations only with the measure $\mu_r^+$. For all measurable sets $E$ we have $\mu_r^+(E) \leq |\mu|(E)$. \\
		Let $u\in \mathcal{GM}_s^{p,q}$ and define the function $g(\xi) = e^{\i u(x) \xi} -1$ analogously to Lemma \ref{ContExp}. Then $g$ is Bochner integrable because of its continuity and taking into account that the measure $\mu_r^+$ is finite. Therefore we obtain the Bochner integral
		\[ \int_{\R} (e^{\i u(x) \xi} -1)\, d\mu_r^+(\xi) = \int_{\R} g(\xi) \, d\mu_r^+(\xi) \]
		with values in $\mathcal{GM}_s^{p,q}$. By applying Minkowski inequality it follows
		\[ \| \int_{\R} (e^{\i u(\cdot) \xi} -1)\, d\mu_r^+(\xi) \|_{\mathcal{GM}_s^{p,q}} \leq \int_{\R} \| e^{\i u(\cdot) \xi} -1 \|_{\mathcal{GM}_s^{p,q}} \, d|\mu|(\xi). \] 
		First suppose that $|\xi|\geq e$. Then Lemma \ref{estSuper} together with equation \eqref{FourierEst} gives
		\begin{eqnarray*}
			\int_{|\xi|\geq e} \| e^{\i u(\cdot) \xi} -1 \|_{\mathcal{GM}_s^{p,q}} \, d|\mu|(\xi) & \leq & C \int_{|\xi|\geq e} e^{b \|\xi u\|^{\frac{1}{s}}_{\mathcal{GM}_s^{p,q}}\log \|\xi u\|_{\mathcal{GM}_s^{p,q}} } \, d|\mu|(\xi) \\
				& < & \infty.
		\end{eqnarray*}
		For $|\xi|\leq e$ the integral obviously converges. \\
		The same estimates also hold for the measures $\mu_r^-$, $\mu_i^+$ and $\mu_i^-$. Thus the result is obtained by
		\begin{eqnarray*}
			 \|\sqrt{2\pi} f(u(x))\|_{\mathcal{GM}_s^{p,q}} & = & \| \int_{\R} g(\xi) \, d\mu_r^+ - \int_{\R} g(\xi) \, d\mu_r^- \\
				& & \qquad + \i \int_{\R} g(\xi) \, d\mu_i^+ - \i \int_{\R} g(\xi) \, d\mu_i^- \|_{\mathcal{GM}_s^{p,q}} \\
				& \leq &  \int_{\R} \| g(\xi) \, d\mu_r^+ \|_{\mathcal{GM}_s^{p,q}} + \int_{\R} \| g(\xi) \, d\mu_r^- \|_{\mathcal{GM}_s^{p,q}} \\
				& & \qquad + \int_{\R} \| g(\xi) \, d\mu_i^+ \|_{\mathcal{GM}_s^{p,q}} + \int_{\R} \| g(\xi) \, d\mu_i^- \|_{\mathcal{GM}_s^{p,q}},
		\end{eqnarray*}
		where every integral on the right-hand side is finite. Thus, the statement is proved.
	\end{proof}
	For practical reasons we remark the following consequence. 
	\begin{cor}
		Let the weight parameter $s>1$ and $\mu$ be a complex measure on $\R$ with the corresponding bounded density function $g$, i.e., $d\mu (\xi) = g(\xi) \, d\xi$. Suppose that
		\begin{equation} \label{densCond}
			\lim_{|\xi|\to \infty} \frac{|\xi|^{\frac{1}{s}} \log |\xi| }{\log |g(\xi)|} = 0
		\end{equation}
		and $\displaystyle \int_{\R} d\mu (\xi) = \int_{\R} g(\xi) \, d\xi = 0$. Assume the function $f$ to be the inverse Fourier transform of $g$. Then $f\in C^\infty$ and the composition operator $T_f: u \mapsto f \circ u$ maps $\mathcal{GM}_s^{p,q}$ into $\mathcal{GM}_s^{p,q}$.
	\end{cor}
	\begin{proof}
		Most of the work has been done in the proof of Theorem \ref{Superposition} where we basically followed \cite{brs}. Here we refer again to \cite{brs}. \\
		The condition \eqref{densCond} yields that the modulus of $\displaystyle \lim_{|\xi|\to\infty} \log |g(\xi)|$ needs to be infinity. This fact together with the boundedness of $g$ gives $\displaystyle \lim_{|\xi|\to\infty} g(\xi)=0$. Moreover, by \eqref{densCond} there exists a sufficiently large number $N>0$ such that 
		\[ - \frac{|\xi|^{\frac{1}{s}} \log |\xi| }{\log |g(\xi)|} \leq \frac{1}{2\lambda} \]
		for all $|\xi|> N$ with $\lambda >0$. Thus, we obtain
		\[ |g(\xi)| \leq e^{-2\lambda |\xi|^{\frac{1}{s}} \log |\xi| } \]
		and it follows
		\begin{eqnarray*}
			\int_{|\xi|>N} e^{\lambda(|\xi|^{\frac{1}{s}} \log |\xi| )} \, d|\mu|(\xi) & = & \int_{|\xi|>N} e^{\lambda(|\xi|^{\frac{1}{s}} \log |\xi| )} |g(\xi)| \, d\xi \hspace{7cm} \\
				& \leq & \int_{|\xi|>N} e^{-\lambda(|\xi|^{\frac{1}{s}} \log |\xi|)} \, d\xi \\
				& < & \infty.
		\end{eqnarray*}
		This completes the proof.
	\end{proof}
	
\section{Application to Partial Differential Equations}
	
	\subsection{A First Result on the Wave Equation}
		In Section \ref{SecGMspace} we obtained some standard tools to treat non-linearities in partial differential equations considered in the frame of Gevrey-modulation spaces. However giving examples of the application of Gevrey-modulation spaces to partial differential equations is beyond the scope of this work. A remark on this is given in Section \ref{SecOpen}. Here we only formulate a first and very basic linear result on the homogeneous wave equation which was already explicitly shown in \cite{reich}. There we did not use the numerous advantages of the Gevrey frame but only classical weighted modulation spaces with weights of Sobolev type. \\
		We consider the initial value problem for the homogeneous wave equation for $(t,x)\in\R^{n+1}$ which is given by 
		\begin{equation}
			\partial_t^2 u(t,x) - \Delta u(t,x) = 0, \quad u(0,x)=f(x), \quad u_t(0,x) = g(x),
		\label{wave}
		\end{equation}
		where $\Delta$ denotes the Laplace-operator 
		\[ \Delta = \frac{\partial^2}{\partial x_1^2} + \frac{\partial^2}{\partial x_2^2} + \ldots + \frac{\partial^2}{\partial x_n^2}. \]
		Applying the Fourier transform we obtain the general solution of (\ref{wave}). It is given by
		\begin{equation} \label{genSol}
			 u(t,x) = \mathcal{F}^{-1}_2 \left(\hat{g}(\xi) \cdot \frac{\sin(|\xi|t)}{|\xi|}\right)(t,x) + \mathcal{F}^{-1}_2 (\hat{f}(\xi) \cdot \cos(|\xi|t))(t,x),
		\end{equation}
		where $\mathcal{F}^{-1}_2$ is the partial inverse Fourier transform in the $\xi$-variable. \\
		In order to determine the space which contains the solution $u=u(t,x)$  we need to introduce the following notation. The space $C([0,T], \mathring{M}^{p,q}_{s,\sigma}(\R^n))$ denotes the function space with the following properties:
		\begin{itemize}
			\item for all $t\in [0,T]$ it holds $u(t, \cdot)\in \mathring{M}^{p,q}_{s,\sigma}(\R^n)$,
			\item $\displaystyle \lim_{t_1\rightarrow t_2} \|u(t_1, \cdot)-u(t_2,\cdot) \|_{\mathring{M}^{p,q}_{s,\sigma}(\R^n)} = 0$ (that is continuity in $t$) and
			\item the norm is defined by $\displaystyle \max_{t\in [0,T]} \|u(t, \cdot)\|_{\mathring{M}^{p,q}_{s,\sigma}(\R^n)}$.
		\end{itemize}
		Moreover, $u\in C^n([0,T], \mathring{M}^{p,q}_{s,\sigma}(\R^n))$ if $\partial_t^l u \in C([0,T], \mathring{M}^{p,q}_{s,\sigma}(\R^n))$ for all $0\leq l\leq n$. Naturally this notation can be applied to every other space with respect to $t$ and $x$, respectively.
	\begin{prop} \label{ModIncCont}
		Assume $s,\sigma$ are real numbers such that $s\geq 0$. Then the modulation space $\mathring{M}^{\infty,1}_{s,\sigma}$ is contained in the space $C$ of all continuous functions, i.e. $\mathring{M}^{\infty,1}_{s,\sigma}(\R^n)\subset C(\R^n)$.
	\end{prop}
	\begin{proof}
		Cf. Proposition 11 in \cite{reich}.
	\end{proof}
	In this work we have established all tools that are needed to prove the following theorem. Thereby we can follow the argumentation given in \cite{reich}. 
	\begin{thm} \label{ModWaveResult}
		Assume that $N$ is an arbitrary integer and $s\in\R$ such that $s\geq 1$. If $f\in \mathring{M}^{p,1}_{s+1,N}(\R^n)$ and $g\in \mathring{M}^{p,1}_{s,N}(\R^n)$, where $p\in [1,\infty]$, then there exists a unique classical solution $u=u(t,x)$ of the Cauchy problem \eqref{wave} such that 
	\[ u \in C([0,T], \mathring{M}^{p,1}_{s+1,N}(\R^n))\bigcap C^1([0,T], \mathring{M}^{p,1}_{s,N}(\R^n))\bigcap C^2([0,T], \mathring{M}^{p,1}_{s-1,N}(\R^n)). \]
		Furthermore the a priori estimate 
	\begin{equation} \label{apriori}
		\|u(t,\cdot)\|_{\mathring{M}^{p,1}_{s+1,N}(\R^n)} \leq C_1(t) \| g \|_{\mathring{M}^{p,1}_{s+1,N}(\R^n)} + C_2(t) \| f \|_{\mathring{M}^{p,1}_{s+1,N}(\R^n)} 
	\end{equation}
	holds for some constants $C_1=C_1(t), C_2=C_2(t) >0$.
	\end{thm}
	\begin{proof}
		Cf. Theorem 7.2 in \cite{reich}.
	\end{proof}
	\begin{rem}
		Obviously the solution $u$ does not loose regularity, that is, for the given initial data $f$ which is contained in the modulation space $\mathring{M}^{p,1}_{s+1,N}$ the solution $u$ also belongs to $\mathring{M}^{p,1}_{s+1,N}$ for every $t\in [0,T]$. \\
		Moreover, this result is independent of the dimension of the physical space, that is the dimension of the $x$-variable. 
		The condition $q=1$ for the modulation space $\mathring{M}^{p,q}_{s+1,N}$ ensures the existence of a classical solution together with the condition $s\geq 1$. If at least one of these conditions are violated, then we still obtain a solution, but not in the classical sense anymore. 
	\end{rem}
	We can easily show that there exists a result for standard modulation spaces, that is for unweighted modulation spaces $\mathring{M}^{p,q}$. 
	\begin{cor}
		Let $f,g\in \mathring{M}^{p,q}(\R^n)$, where $p,q \in [1,\infty]$. Then the Cauchy problem \eqref{wave} has a solution $u \in C([0,T], \mathring{M}^{p,q}(\R^n))$. \\
		Furthermore it holds the a priori estimate \eqref{apriori} but without weights, i.e., $s=\sigma=0$. 
	\end{cor}
	\begin{proof}
		This result immediately follows from the investigations in \cite{reich} which yield Theorem \ref{ModWaveResult}.
	\end{proof}
	\begin{rem}
		Obviously we can neither say something about uniqueness nor about regularity of the derivatives of the solution with respect to $t$.  
	\end{rem}
	
\section{Open Problems and Concluding Remarks} \label{SecOpen}
	
	Due to Section \ref{SecGMspace} we are now able to investigate partial differential equations with non-linearities. Let $L$ be an admissible differential operator. Then it is reasonable to consider problems of the form 
	\begin{equation} \label{PDE} 
		L(u) = f(u), 
	\end{equation}
	where $u\in \mathcal{GM}^{p,q}_s$ and $f$ is an appropriate function. That is we have a non-linear source term. If $f$ is analytic, then Taylor's expansion formula together with Theorem \ref{algebra} give $f(u) \in \mathcal{GM}^{p,q}_s$. Assume $f$ to be as in Theorem \ref{Superposition}, in particular $f$ is non-analytic. Then Theorem \ref{Superposition} yields $f(u) \in \mathcal{GM}^{p,q}_s$. Now we can find methods to solve certain problems of the form \eqref{PDE}. In fact this will be of interest for future work. \\
	In Section \ref{SecGMspace} we obtained convenient statements for Gevrey-modulation spaces. It is natural to ask whether we can reach similar results for classical modulation spaces defined in Definition \ref{modCont}. By following the proof of Theorem \ref{algebra} and taking weights of Sobolev type we can actually prove
		\begin{equation} \label{algMod}
			M^{2p,q}_s \cdot M^{2p,q}_s \subset M^{p,q}_s, 
		\end{equation}
		if $q>1$ and $s>\frac{n}{q'}$, where $q'$ is the conjugated exponent of $q$. In particular under these conditions the modulation space $M^{p,q}_s$ is an algebra under multiplication. For $q=1$ Theorem \ref{multiplication} yields
		\[ M^{p,1} \cdot M^{p,1} \subset M^{p,1}. \] 
		There are comprehensible reasons to state the conjecture that the condition on $s$ in \eqref{algMod} is sharp. This will be proved in future work as well. \\
		Summarizing it seems that we can generally expect some good and fruitful results for the application of modulation spaces to partial differential equations. This fact will mainly motivate our future work on that field.

\end{document}